\documentclass{article}
\usepackage{amsmath,amssymb,amsthm}

\usepackage[dvipdfmx]{graphicx}
\graphicspath{{./picture/}}

\usepackage{url}
\usepackage{color}
\usepackage{here}
\begin{document}

\theoremstyle{plain}
 \newtheorem{thm}{Theorem}[section]
 \newtheorem{lem}[thm]{Lemma}
 \newtheorem{prop}[thm]{Proposition}
 \newtheorem{cor}[thm]{Corollary}
 \newtheorem{cl}[thm]{Claim}
 \newtheorem{conj}[thm]{Conjecture}
 \newtheorem{ques}[thm]{Question}
\theoremstyle{definition}
 \newtheorem{dfn}[thm]{Definition}
\theoremstyle{remark}
 \newtheorem{rem}[thm]{Remark}

\newcommand{\fig}[3][width=12cm]{
\begin{figure}[H]
 \centering 
 \includegraphics[#1,clip]{#2} 
 \caption{#3} 
\label{fig:#2}
\end{figure}}

\newcommand{\blank}{\vspace{0.5\baselineskip}}

\newcommand{\Z}{\mathbb{Z}}
\newcommand{\R}{\mathbb{R}}
\newcommand{\Q}{\mathbb{Q}}
\newcommand{\sT}{T_{\infty}}
\newcommand{\vol}{\mathrm{vol}}
\newcommand{\PSL}{\mathrm{PSL}}

\title{STABLE PRESENTATION LENGTH OF 3-MANIFOLD GROUPS}
\author{KEN'ICHI YOSHIDA}
\date{}

\maketitle

\begin{abstract}
We introduce the stable presentation length of a finitely presentable group. 
The stable presentation length of the fundamental group of a 3-manifold 
can be considered as an analogue of the simplicial volume. 
We show that the stable presentation length have some additive properties  
like the simplicial volume, 
and the simplicial volume of a closed 3-manifold is bounded 
from above and below by constant multiples of 
the stable presentation length of its fundamental group. 
\end{abstract}

\section{Introduction}
\label{section:introduction}

Mostow-Prasad rigidity~\cite{mostow1973strong, prasad1973strong} states that 
a finite volume hyperbolic 3-manifold is determined 
by its fundamental group. 
In particular, the volume of a hyperbolic 3-manifold is a topological invariant. 
The simplicial volume of a manifold introduced 
by Gromov~\cite{gromov1982volume} is defined topologically, 
which is proportional to the volume for a hyperbolic manifold. 
Furthermore, the simplicial volume of a Seifert 3-manifold is equal to zero 
and the simplicial volume of a 3-manifold is additive 
for connected sums and decompositions along incompressible tori 
\cite{soma1981gromov}. 
Therefore, the geometrization theorem proved by 
Perelman~\cite{perelman2002entropy, perelman2003ricci} 
implies that the simplicial volume of an orientable closed 3-manifold 
is equal to the sum of the simplicial volumes of hyperbolic pieces 
after the geometrization. 

The simplicial volume of a closed 3-manifold is uniquely determined 
by its fundamental group. 
If the fundamental group of an orientable closed 3-manifold 
is a freely decomposable, the 3-manifold can be decomposed into 
a connected sum corresponding to the free product. 
Hence it is sufficient to show that 
the claim holds for closed irreducible 3-manifolds. 
A closed Haken 3-manifold is determined 
by its fundamental group \cite{waldhausen1968irreducible}. 
A non-Haken 3-manifold is elliptic or hyperbolic by the geometrization. 
The simplicial volume of an elliptic manifold is equal to zero 
and Mostow rigidity implies that 
a hyperbolic manifold is determined by its fundamental group. 
In order to consider a direct relation between 
the simplicial volume of a 3-manifold and its fundamental group, 
we will introduce 
the stable presentation length of a finitely presentable group.

Milnor and Thurston~\cite{milnor1977characteristic} considered 
some characteristic numbers of manifolds, 
where ``characteristic'' means multiplicativity for 
the finite sheeted coverings, 
i.e. an invariant $C$ of manifolds is a characteristic number 
if it holds that $C(N) = d \cdot C(M)$ for any $d$-sheeted covering $N \to M$. 
For example, the Euler characteristic and the simplicial volume 
are characteristic numbers. 
We say such an invariant is \textit{volume-like} instead of 
a characteristic number in order to indicate similarity to the volume. 
Milnor and Thurston introduced the following volume-like invariant 
of a manifold, which is called the stable $\Delta$-complexity by 
Francaviglia , Frigerio and Martelli~\cite{francaviglia2012stable}. 
The \textit{$\Delta$-complexity} $\sigma (M)$ of a closed $n$-manifold $M$ 
is the minimal number of $n$-simplices in a triangulation of $M$. 
In this paper we use the term ``triangulation'' in a weak sense, i.e. 
a triangulation of a manifold $M$ is a cellular decomposition of $M$ 
such that each cell is a simplex. 
The $\Delta$-complexity is not volume-like, but an upper volume 
in the sense of Reznikov~\cite{reznikov1996volumes}, i.e. 
it holds that $\sigma (N) \leq d \cdot \sigma(M)$ 
for any $d$-sheeted covering $N \to M$. 
Then a natural way gives a volume-like invariant 
defined by 
\[
\sigma _{\infty} (M) = \inf _{N \to M} \frac{\sigma (N)}{\deg (N \to M)}, 
\]
where the infimum is taken over the finite sheeted coverings of $M$. 
$\sigma _{\infty} (M) $ is called the stable $\Delta$-complexity of $M$.

While the stable $\Delta$-complexity is hard to handle, 
the simplicial volume following it can work similarly and has more application. 
Thus the stable $\Delta$-complexity became something obsolete, 
but recently Francaviglia, Frigerio and Martelli~\cite{francaviglia2012stable} 
brought a further development. 
They introduced the stable complexity of a 3-manifold. 
The \textit{complexity} $c(M)$ of 3-manifold $M$ is the minimal number of vertices 
in a simple spine for $M$. 
Matveev~\cite[Theorem 5]{matveev1990complexity} showed that 
the complexity of $M$ is 
equal to its $\Delta$-complexity if $M$ is irreducible and 
not $S^3, \mathbb{RP}^{3}$ or the lens space $L(3,1)$. 
In particular, the two complexities of $M$ coincide 
if $M$ is a hyperbolic 3-manifold. 
The stable complexity $c_{\infty} (M)$ is defined 
in the same way as the stable $\Delta$-complexity. 
Francaviglia, Frigerio and Martelli 
showed that the stable complexity has same additivity 
as the simplicial volume of 3-manifold, 
and therefore $c_{\infty} (M)$ is the sum of the ones 
of hyperbolic pieces after the geometrization 
\cite[Corollary 5.3 and Proposition 5.10]{francaviglia2012stable}. 
Moreover, the stable complexity of 3-manifold is bounded 
from above and below by constant multiples of the simplicial volume. 
This is implied from the fact that the stable $\Delta$-complexity of a hyperbolic 3-manifold is so.

Delzant~\cite{delzant1996decomposition} introduced a complexity $T(G)$ of a finitely presentable group $G$. 
We call it the presentation length according to 
Agol and Liu~\cite{agol2012presentation}. 
Delzant also introduced a relative version of presentation length, 
and he gave an estimate of presentation length for a decomposition of group. 
There are some applications for the presentation length 
of the fundamental group of a 3-manifold. 
Cooper~\cite{cooper1999volume} gave an upper bound 
for the volume of a hyperbolic 3-manifold by the presentation length. 
White~\cite{white2001diameter} gave an estimate 
for the diameter of a closed hyperbolic 3-manifold by the presentation length. 
Agol and Liu~\cite{agol2012presentation} solved Simon's conjecture 
by using presentation length.

Delzant and Potyagailo~\cite{delzant2013complexity} remarked that the volume 
of hyperbolic 3-manifold is not bounded from below by a constant multiple 
of the presentation length. 
They considered a relative presentation length for a thick part 
of a hyperbolic 3-manifold, and showed that the volume is bounded 
from above and below by constant multiples of this relative presentation length. 
We will introduce the stable presentation length instead of this. 

The presentation length is an upper volume. Hence we can define 
the stabilization of the presentation length. 
We will show the stable presentation length of a 3-manifold 
has additivity like the simplicial volume and the stable complexity
(Theorem \ref{thm:additivityfree} and Theorem \ref{thm:additivityJSJ}).

We conjecture that the stable presentation length for a 3-manifold 
is half of the stable complexity (Conjecture \ref{conj}). 
This conjecture is relevant to the following problem. 
Francaviglia, Frigerio and Martelli gave a problem asking whether 
the simplicial volume and the stable complexity of a 3-manifold coincide. 
The additivity of the simplicial volume and the stable complexity 
reduces this problem to the cases for the hyperbolic 3-manifolds 
\cite[Question 6.5]{francaviglia2012stable}. 
The Ehrenpreis conjecture proved by Kahn and Markovic~\cite{kahn2015good} 
states that for any two closed hyperbolic surfaces $M,N$ and $K>1$ 
there are finite coverings of $M,N$ which are $K$-quasiconformal. 
The simplicial volume and the stable $\Delta$-complexity 
of a hyperbolic 3-manifold coincide 
if and only if, roughly speaking, a hyperbolic 3-manifold has 
a finite covering with a triangulation 
in which almost all the tetrahedra after straightening are nearly isometric
to an ideal regular tetrahedron. 
Therefore the above problem can be considered as 
a 3-dimensional version of Ehrenpreis problem.

Frigerio, L\"{o}h, Pagliantini and Sauer~\cite{frigerio2016integral} 
showed that the simplicial volume and the stable integral simplicial volume 
of a closed hyperbolic 3-manifold coincide. 
The integral simplicial volume of an oriented closed manifold is 
defined as the seminorm of the fundamental class in the integer homology. 
The stable integral simplicial volume is 
the stabilization of the integral simplicial volume 
in the same way as the stable $\Delta$-complexity. 
Since the integral simplicial volume is quite similar to 
the $\Delta$-complexity, 
the above result supports the affirmative answer 
to that 3-dimensional version of Ehrenpreis problem 
at least for the closed hyperbolic 3-manifolds. 
In contrast to the lower dimensional cases, 
the simplicial volume and the stable integral simplicial volume 
of a closed hyperbolic manifold of dimension at least 4 cannot coincide 
\cite[Theorem 2.1]{francaviglia2012stable}.

The presentation length of a group can be considered 
as a two-dimensional version of the rank, 
which is the minimal number of generators. 
The relation between the presentation length and the $\Delta$-complexity 
of a 3-manifold 
is analogous to the relation between the rank and the Heegaard genus. 
The Heegaard genus of a closed 3-manifold 
is not less than the rank of its fundamental group, 
and they do not coincide in general 
\cite{boileau1984heegaard, li2013rank}. 
Lackenby~\cite{lackenby2005expanders, lackenby2006heegaard} introduced the rank gradient and the Heegaard gradient 
to approach the virtually Haken conjecture. 
The rank gradient of a finitely generated group $G$ is defined as 
$\inf_{H} (\mathrm{rank}(H) -1)/[G:H]$, 
where the infimum is taken over all the finite index subgroups $H$ of $G$. 
Similarly, 
The Heegaard gradient of a finitely generated group $G$ is defined as 
$\inf_{N} \chi_{-}^{h} (N)/\deg (N \to M)$, 
where $\chi_{-}^{h} (N)$ is the negative of the maximal Euler characteristic 
of a Heegaard surface of $N$ 
and the infimum is taken over all the finite coverings $N$ of $M$. 
Since the virtually fibered conjecture is solved 
by Agol~\cite{agol2013virtual}, 
we know that the rank gradient and the Heegaard gradient of 
a hyperbolic 3-manifold are equal to zero.

We mention a relation between the homology torsion 
and the stable presentation length. 
Let $|\mathrm{Tor}(A)|$ denote 
the order of the torsion part of an abelian group $A$. 
For a group $G$, 
We consider the torsion part of its abelianization $G/[G,G]$ 
(in other words, the first integral homology). 
Pervova and Petronio~\cite{pervova2008complexity} gave the following inequality: 
If $G$ is a finitely presentable group without 2-torsion, 
it holds that 
\[
T(G) \geq \log_{3} |\mathrm{Tor}(G/[G,G])|. 
\] 
As a relevant problem, 
Bergeron, Venkatesh, Le and 
L\"{u}ck~\cite{bergeron2013asymptotic, le2014growth, luck2013approximating}
conjectured 
\[
\lim_{i} \frac{\log |\mathrm{Tor}(G_{i}/[G_{i},G_{i}])|}{[G:G_{i}]} 
= \frac{\vol (M)}{6\pi} 
\]
for a hypervolic 3-manifold $M$ and 
an appropriate sequence $G_{1} > G_{2} > \dots$ 
of finite index subgroups of $G = \pi_{1} (M)$. 
Bergeron and Venkatesh gave conjectures 
also for lattices in more general Lie groups.

At last we give a question for amenable groups. 
The simplicial volume of a manifold of amenable fundamental group 
is equal to zero \cite[Section 3.1]{gromov1982volume}, 
and the rank gradient of a finitely presentable, residually finite, 
infinite amenable group is also equal to zero 
\cite[Theorem 1.2]{lackenby2005expanders}. 
Similarity between the volume and the stable presentation length 
induce the following question. 
\begin{ques}
For a finitely presentable amenable group $G$, 
is the stable presentation length $\sT (G)$ equal to zero? 
\end{ques}

\subsection*{Organization of the paper}

\indent

In Section \ref{section:preliminaries}, we review the definition and elementary 
properties of the presentation length. 

In Section \ref{section:definition}, we define the stable presentation length 
as a volume-like invariant of a finitely presentable group. 

In Section \ref{section:hyperbolic}, we consider the stable presentation length 
of a hyperbolic 3-manifold. 
For a 3-manifold with boundary, it is natural to consider 
its presentation length relative to the fundamental groups of the boundary component. 
we show that the stable presentation length of the hyperbolic 3-manifold 
relative to the cusp subgroups coincides the non-relative stable presentation length (Theorem \ref{thm:relabshyp}). 
In fact, we show a more general result for a residually finite group and 
free abelian subgroups (Theorem \ref{thm:relabs}). 
This result is the most technical part in this paper. 
The simplicial volume has a similar property 
\cite[Theorem 1.5]{loh2009simplicial}. Namely, 
We can consider two versions of simplicial volume of a manifold $M$ with boundary. 
One is the seminorm of the relative fundamental class, 
and another is for the open manifold $\mathrm{int}M$. 
They coincide if the fundamental groups of the boundary components 
are amenable. 
Furthermore, we show that the stable presentation length of a hyperbolic 
3-manifold is bounded by constant multiples of the volume and the stable complexity.

In Section \ref{section:additivity}, 
we show additivity of the stable presentation length. 
We give a proof as with the proof for the stable complexity 
by using Delzant's result (Theorem \ref{thm:lowerbound}) 
and Theorem \ref{thm:relabs}. 
We also show that the stable presentation length 
of a Seifert 3-manifold vanishes (Theorem \ref{thm:seifertvanish}). 
These results imply that the stable presentation length of a closed 3-manifold 
is equal to the sum of the stable presentation lengths of hyperbolic pieces 
after the geometrization.

In Section \ref{section:example}, 
we give some examples of stable presentation length. 
The stable presentation lengths of the surface groups are the only example 
of non-zero explicit value of stable presentation length in this paper. 
We also give examples for fundamental groups of some hyperbolic 3-manifolds. 
Those examples support Conjecture \ref{conj}. 

\subsection*{Acknowledgements} 

The author would like to express his gratitude to Professor Takashi Tsuboi 
for helpful suggestions and stimulating discussions. 
He would also like to thank Professor Ian Agol and Professor Toshitake Kohno 
for their warm encouragement. 
This research is supported by 
JSPS Research Fellowships for Young Scientists 
and the Program for Leading Graduate Schools, MEXT, Japan.

\section{Preliminaries for presentation length}
\label{section:preliminaries}

We review the definition of presentation length 
(also known as Delzant's $T$-invariant) and some elementary facts. 
See Delzant~\cite{delzant1996decomposition} for details. 

\begin{dfn}
\label{dfn:T}
Let $G$ be a finitely presentable group. 
We define the \textit{presentation length} $T(G)$ of $G$ by 
\[
T(G) = \min_{\mathcal{P}} \sum_{i=1}^{m} \max \{0, |r_{i}| -2\}, 
\] 
where we take the minimum for the presentations 
such as \\ $\mathcal{P} = \langle x_1, \dots , x_n | r_1, \dots , r_m \rangle$ of $G$, 
and let $| r_i |$ denote the word length of $r_i$. 
\end{dfn}

We associate the \textit{presentation complex} $P$ to a presentation \\
$\mathcal{P} = \langle x_1, \dots , x_n | r_1, \dots , r_m \rangle$ of $G$. 
$P$ is the 2-dimensional cellular complex consisting of a single 0-cell, 
1-cells and 2-cells corresponding to the generators and relators. 
Then $\pi _1 (P)$ is isomorphic to $G$. 
By dividing a $k$-gon of a presentation complex into $k-2$ triangles, 
$T(G)$ can be realized by a \textit{triangular presentation} of $G$, 
i.e. a presentation $\langle x_1, \dots , x_n | r_1, \dots , r_m \rangle$ 
in which each word length $|r_i|$ is equal to 2 or 3. 
If $G$ has no 2-torsion, we can assume $|r_i| = 3$. 
From now on, a presentation complex is always assumed to be triangular, 
i.e. each of its 2-cells is a triangle or a bigon. 
$T(G)$ is the minimal number of triangles in a presentation complex for $G$. 

Delzant~\cite{delzant1996decomposition} also introduced a relative version 
of the presentation length. We need this in order to estimate 
the presentation length under a decomposition of group. 
Before defining the relative presentation length, 
we prepare the notion of an orbihedron 
due to Haefliger~\cite{haefliger1991complexes}, 
which is an analogue of an orbifold. 
An \textit{orbihedron} is a cellular complex with isotropy groups on cells 
such that 
after an appropriate subdivision of cells, 
the star neighborhood of each cell $c$ is isomorphic to 
the quotient of a cellular complex 
by a certain cellular action of the isotropy group 
which fixes the preimage of $c$. 
This local structure gives the notion of a covering space of an orbihedron, 
analogously to an orbifold. 
The universal covering of an orbihedron is a covering 
without further nontrivial connected coverings. 
The fundamental group of an orbihedron is the deck transformation 
of its universal covering. 
Consequently, an orbihedron is isomorphic to 
the quotient of its universal covering by its fundamental group. 
If the isotropy group on every cell in the universal covering 
of an orbihedron is trivial, 
its isotropy groups are identified with 
subgroups of the fundamental group up to conjugacy. 
Note that isotropy groups of an orbihedron are possibly infinite, 
unlike an orbifold. 

\begin{dfn}
\label{dfn:relT}
Let $G$ be a finitely presentable group. 
Suppose that $C_{1}, \dots ,C_{l}$ are subgroups of $G$. 
A \textit{(relative) presentation complex} $P$ 
for $(G; C_{1}, \dots ,C_{l})$
is a 2-dimensional orbihedron satisfying the following conditions: 
\begin{itemize}
 \item Any 2-cell of $P$ is a triangle or a bigon. 
 \item The 0-cells of $P$ consist of $l$ vertices 
       with isotropy groups $C_{1}, \dots ,C_{l}$. 
       The isotropy groups of the 1-cells and 2-cells are trivial. 
 \item The isotropy groups of the universal covering of $P$ are trivial, 
       and the fundamental group $\pi _{1}^{\mathrm{orb}}(P)$ of $P$ 
       as an orbihedron is isomorphic to $G$. 
       This isomorphism makes the isotropy groups $C_{1}, \dots ,C_{l}$ 
       be subgroups of $G$ up to conjugacy. 
\end{itemize}
We define the \textit{relative presentation length} 
$T(G; C_{1}, \dots ,C_{l})$ as the minimal number of triangles 
in a relative presentation complex for $(G; C_{1}, \dots ,C_{l})$. 
We say that a presentation complex $P$ is \textit{minimal} 
if $P$ realizes the presentation length. 
\end{dfn}

Our definition requires that the isotropy is only on the vertices, 
but this is not essential. 
Indeed, if isotropy of a 2-complex is on edges or 2-cells, 
we can construct a presentation complex by replacing edges with bigons. 
We can consider only the conjugacy classes of $C_{1}, \dots ,C_{l} < G$. 
By definition, we have $T(G; \{ 1 \}) = T(G)$. 
We can allow a presentation complex for $G$ to have more than one vertex, 
namely, $T(G; \{ 1 \}, \dots , \{ 1 \}) = T(G; \{ 1 \})$. 
This follows by contracting vertices of a presentation complex along edges, 
without changing the fundamental group. 
More generally, the following holds. 

\begin{prop}
\emph{\cite[Lemma I.1.3]{delzant1996decomposition}}
For a finitely presentable group $G$ 
and its subgroups $C, C^{\prime}, C_{1}, \dots ,C_{l}$, 
suppose that $C^{\prime}$ is contained in a conjugate of $C$. 
Then 
\[
T(G; C, C^{\prime}, C_{1}, \dots ,C_{l}) = T(G; C, C_{1}, \dots ,C_{l}). 
\]
\end{prop}

The relative presentation length is finite in a usual case 
though it was not declared. 
The construction in the proof will be used for the proof of Theorem \ref{thm:relabs}. 

\begin{prop}
\label{prop:cone}
Let $G$ be a finitely presentable group. 
Suppose that $C_{1}, \dots ,C_{l}$ are finitely generated subgroups of $G$. 
Then there is a finite presentation complex for 
$(G; C_{1}, \dots ,C_{l})$, 
in other words, we have $T(G; C_{1}, \dots ,C_{l}) < \infty$. 
\end{prop}

\begin{proof}
Take a presentation complex $P$ for $G$. 
Let $y_{i1} , \dots , y_{ik_{i}}$ be generators of $C_{i}$ for $a \leq i \leq l$. 
There exist simplicial paths $a_{i1} , \dots , a_{ik_{i}}$ in $P$ 
corresponding to $y_{i1} , \dots , y_{ik_{i}}$. 
We construct a complex $P^{\prime}$ 
by attaching cones of $a_{i1} , \dots , a_{ik_{i}}$ to $P$ 
(Figure \ref{fig:spl-cone}). 
Put isotropy $C_{i}$ on the vertex of the $i$-th cone. 
Then $P^{\prime}$ is a finite presentation complex 
for $(G; C_{1}, \dots ,C_{l}, \{ 1 \})$. 
\end{proof}

\fig[width=6cm]{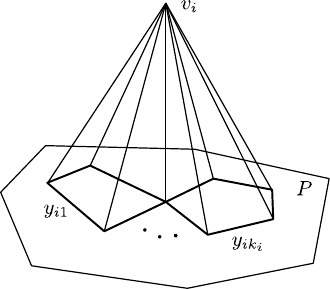}{Construction of a relative presentation complex}

Delzant~\cite{delzant1996decomposition} show how the presentation length behaves 
under a decomposition into a graph of groups. 
A \textit{graph of groups} $\mathcal{G}$ in the sense of Serre~\cite{serre1980trees} 
is a collection of the following data: 
\begin{itemize}
 \item An underlying connected graph $\Gamma$, 
       consisting a vertex set $V$, an edge set $E$ and maps $o_{\pm} \colon E \to V$ 
       from edges to their end points. 
 \item Vertex groups $\{ G_v \}$ and edge groups $\{ C_e \}$ 
       for $v \in V$ and $e \in E$. 
 \item Injections $\{ \iota _{\pm} \colon C_e \hookrightarrow G_{o_{\pm}(e)} \}$ 
       for $e \in E$. 
\end{itemize}
The fundamental group $\pi_{1} (\mathcal{G})$ can be characterized 
as the following. 
A graph of spaces $\mathcal{X}$ corresponding to $\mathcal{G}$ 
is a collection of CW-complexes $\{ X_{v} \}, \{ X_{e} \}$ 
and $\pi_{1}$-injective maps 
$\{i_{\pm} \colon X_{e} \to X_{o_{\pm}(e)} \}$, 
where $\pi_{1} (X_{v}) = G_{v}, \pi _1 (X_{e}) = C_{e}$ 
and $i_{\pm}$ induces $\iota _{\pm}$. 
We construct a space 
\[
X_\mathcal{X} = 
\left( \coprod_{v \in V} X_{v} \sqcup \coprod_{e \in E} (X_{e} \times [-1, 1])\right) /\sim ,
\]
where the gluing relation is that $(x, \pm 1) \sim i_{\pm}(x) $ for $x \in X_e$. 
Then $\pi _1 (\mathcal{G}) = \pi _1 (X_\mathcal{X})$. 
For a given group $G$, we say that $\mathcal{G}$ is a decomposition of $G$ 
if $G \cong \pi_{1} (\mathcal{G})$. 

Let $\mathcal{G}$ be a decomposition of a group $G$. 
Suppose that $G_{1} , \dots , G_{n}$ are the vertex groups of $\mathcal{G}$ 
and $C_{1} , \dots , C_{l}$ are the edge groups of $\mathcal{G}$. 
We construct presentation complexes $P_i$ for $(G_{i} ; C_{i1}, \dots , C_{il_{i}})$, 
where $C_{ij}$ for $1 \leq j \leq l_{i}$ are the edge groups corresponding to the edges 
such that the $i$-th vertex is its end point. 
We can construct a presentation complex $P$ for $(G; C_1 , \dots , C_l)$ 
by gluing $P_1 , \dots , P_n$ along their vertices. 
Then the number of the triangles of $P$ is the sum of the ones of $P_i$. 
Therefore we have the following proposition. 

\begin{prop}
\label{prop:upperbound}
\emph{\cite[Lemma I.1.4]{delzant1996decomposition}}
Let $G, C_i$ and $C_{ij}$ be as above. Then 
\[
T(G; C_1 , \dots , C_l) \leq \sum_{i=1}^{n} T(G_i ; \{C_{ij}\}_{1 \leq j \leq l_{i}}). 
\]
\end{prop}

We need to consider a ``good'' decomposition 
in order to estimate the presentation length from below. 

\begin{dfn} 
\label{dfn:regid}
Let $\mathcal{G}$ be a decomposition of $G$, 
and let $C_1 , \dots , C_l$ be the edge subgroups of $\mathcal{G}$. 
A subgroup $C$ of $G$ is \textit{rigid} 
if it satisfies the following condition: 
If $G$ acts a tree $T$ without inversion 
and $C$ contains a nontrivial stabilizer of an edge of $T$, 
$C$ fixes a vertex of $T$. 
$\mathcal{G}$ is \textit{rigid} 
if every edge group of $\mathcal{G}$ is rigid. 

Let $C_{ij}$ be as Proposition \ref{prop:upperbound}. 
$\mathcal{G}$ is \textit{reduced} 
if there is no decomposition $\mathcal{G^{\prime}}$ of $G_i$ such that 
$C_{ij}$ is a vertex group of $\mathcal{G^{\prime}}$, for any $G_i$ and $C_{ij}$. 
\end{dfn}

Under the above preparation, we can state the following highly nontrivial fact. 

\begin{thm}
\label{thm:lowerbound}
\emph{\cite[Theorem II]{delzant1996decomposition}}
Let $\mathcal{G}, G_i$ and $C_{ij}$ be as Proposition \ref{prop:upperbound}. 
Suppose that $\mathcal{G}$ is rigid and reduced. 
Then 
\[
T(G) \geq \sum_{i=1}^{n} T(G_i ; \{C_{ij}\}_{1 \leq j \leq l_{i}}). 
\]
\end{thm}

Since a free product decomposition of a group is rigid and reduced, 
we have the following theorem. 

\begin{cor}
\label{cor:free}
\emph{\cite[Corollary I]{delzant1996decomposition}}
Let $G = A*B$ be a free product of finitely presentable groups. 
Then $T(G) = T(A) + T(B)$. 
\end{cor}

We will mainly consider the fundamental group of an orientable 3-manifold. 
A connected sum decomposition of a 3-manifold induces 
a free product decomposition of the fundamental group, 
which concerns Corollary~\ref{cor:free}. 
A (possibly disconnected) orientable surface $S$ embedded 
in an irreducible orientable 3-manifold $M$ 
is an \textit{essential surface} 
if $S$ does not contain a sphere, 
each component $S_{i}$ of $S$ induces an injection 
$\pi_{1}(S_{i}) \hookrightarrow \pi_{1}(M)$, 
and no pair of components of $S$ are parallel. 
A decomposition of an irreducible orientable 3-manifold 
along an essential surface 
induces a decomposition of its fundamental group into a graph of group. 
Then a component of the decomposed manifold corresponds to a vertex group, 
and a component of the essential surface corresponds to an edge group. 
We can apply Theorem \ref{thm:lowerbound} in this case.

\begin{prop}
\label{prop:rigid}
\emph{\cite[Proposition I.6.1]{delzant1996decomposition}}
Let $\mathcal{G}$ be a decomposition of the fundamental group 
of an irreducible orientable 3-manifold $M$. 
Suppose that $\mathcal{G}$ corresponds to a decomposition of $M$ along  
an essential surface. 
Then $\mathcal{G}$ is rigid and reduced. 
\end{prop}

\section{Definition of stable presentation length}
\label{section:definition}

The (relative) presentation length is an upper volume, 
i.e. it has the following sub-multiplicative property. 

\begin{prop}
\label{prop:upvol}
For a finitely presentable group $G$, let $H$ be a finite index subgroup of $G$. 
Let $d = [G:H]$ denote the index of $H$ in $G$. 
Suppose that $C_1 , \dots , C_l$ are subgroups of $G$. 
Then 
\[
T(H; \{ gC_{i}g^{-1} \cap H \}_{1 \leq i \leq l, g \in G}) 
\leq d \cdot T(G; C_1 , \dots , C_l). 
\]
In particular, $T(H) \leq d \cdot T(G)$. 
\end{prop}

We remark that $\{ gC_{i}g^{-1} \cap H \}_{1 \leq i \leq l, g \in G}$ is a finite family of subgroups up to conjugate in $H$, 
since $H$ is a finite index subgroup of $G$. 

\begin{proof}
Let $P$ be a minimal presentation complex for $(G; C_1 , \dots , C_l)$. 
There exists a $d$-sheeted covering $\widetilde{P}$ of $P$ as an orbihedron 
which corresponds to $H \leq G$. 
Then the isotropies on the vertices of $\widetilde{P}$ 
are $\{ gC_{i}g^{-1} \cap H \}_{1 \leq i \leq l, g \in G}$. 
Therefore $\widetilde{P}$ a presentation complex 
for $(H; \{ gC_{i}g^{-1} \cap H \}_{1 \leq i \leq l, g \in G})$ with $d \cdot T(G; C_1 , \dots , C_l)$ triangles. 
\end{proof}

Proposition \ref{prop:upvol} leads to the definition of stable presentation length 
as an analogue of the stable $\Delta$-complexity 
by Milnor and Thurston~\cite{milnor1977characteristic}. 
Stable presentation length is a ``volume-like'' invariant, i.e. 
it is multiplicative for finite index subgroups. 

\begin{dfn}
\label{dfn:stableT}
We define the \textit{stable presentation length} $\sT (G)$ 
of a finitely presentable group $G$ by 
\[
\sT (G) = \inf_{H \leq G} \frac{T(H)}{[G:H]}, 
\]
where the infimum is taken over all the finite index subgroups $H$. 
Furthermore, suppose that $C_1 , \dots , C_l$ are subgroups of $G$.  
We define the \textit{(relative) stable presentation length} as 
\[
\sT (G; C_1 , \dots , C_l) = 
\inf_{H \leq G} \frac{T(H; \{ gC_{i}g^{-1} \cap H \}_{1 \leq i \leq l, g \in G})}{[G:H]}. 
\]
\end{dfn}

\begin{prop}
\label{prop:vol}
Let $G, H, d$ and $C_1 , \dots , C_l$ be as Proposition \ref{prop:upvol}. 
Then 
\[
\sT (H; \{ gC_{i}g^{-1} \cap H \}_{1 \leq i \leq l, g \in G}) 
= d \cdot \sT (G; C_1 , \dots , C_l). 
\]
In particular, $\sT (H) = d \cdot \sT (G)$. 
\end{prop}

\begin{proof}
Take a finite index subgroup $G^{\prime}$ of $G$. 
Then 
$H^{\prime} = G^{\prime} \cap H$ is also a finite index subgroup of $G$. 
We have 
\[
T (H^{\prime}; \{ gC_{i}g^{-1} \cap H^{\prime} 
\}_{1 \leq i \leq l, g \in G})
\leq 
[G^{\prime} : H^{\prime}] T (G^{\prime}; \{ gC_{i}g^{-1} \cap G^{\prime}
\}_{1 \leq i \leq l, g \in G})
\]
by Proposition \ref{prop:upvol}. 
Hence we can calculate $\sT (G; C_1 , \dots , C_l)$ by taking the infimum 
for only the subgroups of $H$. 
Therefore 
\begin{align*}
\sT (H; \{ gC_{i}g^{-1} \cap H \}_{1 \leq i \leq l, g \in G}) 
& = 
\inf_{H^{\prime} \leq H} \frac{T(H^{\prime}; 
\{ gC_{i}g^{-1} \cap H^{\prime} \}_{1 \leq i \leq l, g \in G})}{[H:H^{\prime}]} \\
& =  
d \cdot \inf_{H^{\prime} \leq H} \frac{T(H^{\prime}; 
\{ gC_{i}g^{-1} \cap H^{\prime} \}_{1 \leq i \leq l, g \in G})}{[G:H^{\prime}]} \\
& = 
d \cdot \sT (G; C_1 , \dots , C_l). 
\end{align*}
\end{proof}

\section{Stable presentation length for hyperbolic 3-manifolds}
\label{section:hyperbolic}

We consider the stable presentation length of the fundamental group 
of a compact 3-manifold $M$. 
We write 
\begin{align*}
T(M) = T(\pi _1 (M)), \quad \sT (M) = \sT (\pi _1 (M)), \\
T(M; \partial M) = T(\pi _1 (M); \pi _1 (S_1), \dots , \pi _1 (S_l)), \\ 
\sT (M; \partial M) = \sT (\pi _1 (M); \pi _1 (S_1), \dots , \pi _1 (S_l)), 
\end{align*}
where $S_1, \dots , S_l$ are the components of $\partial M$. 
We call them the (relative, stable) presentation length of $M$ respectively. 

If $M$ is a 3-manifold with boundary, we can also consider 
the relative presentation length $T(M; \partial M)$. 
For instance, let $M$ be a finite volume cusped hyperbolic 3-manifold. 
We consider $M$ as a compact 3-manifold with boundary. 
The interior of $M$ admits a hyperbolic metric. 
Let $S_1, \dots , S_l$ be the components of $\partial M$. 
The 2-skeleton of an ideal triangulation of $M$ 
(i.e. a cellular decomposition of the space obtained by smashing each boundary component of $M$ to a point such that every 3-cell is tetrahedron and 
its vertices are the points from boundary components of $M$) 
can be regarded as a relative presentation complex of 
$(\pi _1 (M); \pi _1 (S_1), \dots , \pi _1 (S_l))$. 
We show that this relative stable presentation length coincides with the absolute 
stable presentation length. 

\begin{thm}
\label{thm:relabshyp}
For a finite volume hyperbolic 3-manifold $M$, it holds that 
$\sT (M; \partial M) = \sT (M)$. 
\end{thm}

More generally, we show the following theorem. 
Since $\pi _1 (M)$ is linear for a hyperbolic 3-manifold $M$, 
it is residually finite \cite{hempel1987residual}. 

\begin{thm}
\label{thm:relabs}
Let $G$ be a finitely presentable group, 
and let $C_{1}, \dots , C_{l}$ be free abelian subgroups of $G$ 
whose ranks are at least two. 
Suppose $G$ is residually finite. 
Then it holds that $\sT (G; C_{1}, \dots , C_{l}) = \sT (G)$. 
\end{thm}

We remark that it is necessary to suppose the rank of $C_{i}$ is at least two. 
The inequality does not hold for the case of Theorem \ref{thm:cuspedsurface}, 
since $\sT (\pi _1 (\Sigma _{g,b})) =0$. 

For an integer $p>1$, the \textit{$p$-characteristic covering} of torus $T^2$ is 
the covering which corresponds to the subgroup 
$p\Z \times p\Z < \Z \times \Z \cong \pi _1 (T^2)$. 
A \textit{$p$-characteristic covering} of $M$ is a finite covering 
whose restriction on each cusp is a union of $p$-characteristic coverings of torus. 
A hyperbolic 3-manifold $M$ admits $p$-characteristic coverings 
for arbitrarily large $p$ \cite[Lemma 4.1]{hempel1987residual}. 
We can use them for a proof of Theorem \ref{thm:relabshyp}. 
In general, however, a residually finite group $G$ with $C_{1}, \dots , C_{l}$
may not have such subgroups. 
Nonetheless, we can take a nearly orthogonal basis of subgroup of $C_{i}$ 
with respect to a basis of $C_{i}$ $(1 \leq i \leq l)$. 

A \textit{lattice} in $\R ^{n}$ is a discrete subgroup of $\R ^{n}$ 
which spans $\R ^{n}$. 
A lattice in $\R ^{n}$ has a nearly orthogonal basis 
as in Lemma~\ref{lem:reducedbasis}, called a \textit{reduced basis}. 
We refer to Cassels \cite[Ch.VIII.5.2]{cassels1971introduction} 
for a proof. 
Lenstra, Lenstra and Lov\'{a}sz~\cite{lenstra1982factoring} gave 
a polynomial time algorithm to find a reduce basis. 
We will use the following lemma with a 1-norm on $\R ^{n}$. 

\begin{lem}
\label{lem:reducedbasis}
Given a norm $\| \cdot \|$ in $\R ^{n}$, 
there is a constant $\epsilon _{n}$ such that the following holds. 
If $\Lambda$ is a lattice in $\R ^{n}$, 
then there is a basis $(v_{1}, \dots , v_{n})$ of $\Lambda$ 
such that 
\[
d(\Lambda) \geq \epsilon _{n} \|v_{1}\| \cdots  \|v_{n}\|, 
\]
where $d(\Lambda)$ is the covolume of $\Lambda$, which is the determinant 
of the matrix whose columns are $v_{i}$'s. 
\end{lem}

\begin{proof}[proof of Theorem \ref{thm:relabs}]
We first show that $\sT (G; C_{1}, \dots , C_{l}) \leq \sT (G)$. 

Assume that $\sT (G; C_{1}, \dots , C_{l}) \leq T (G)$ 
for any groups $G$ and $C_{1}, \dots , C_{l}$ satisfying the condition. 
Then 
$\sT (H; \{ gC_{i}g^{-1} \cap H \}_{1 \leq i \leq l, g \in G}) \leq T(H)$ 
for a finite index subgroup $H$ of $G$. 
Proposition~\ref{prop:vol} implies that 
$\sT (G; C_1 , \dots , C_l) \leq T(H)/[G:H]$. 
By taking the infimum over $H$, 
we obtain $\sT (G; C_{1}, \dots , C_{l}) \leq \sT (G)$. 
Hence it is sufficient to show that 
$\sT (G; C_{1}, \dots , C_{l}) \leq T (G)$. 
For simplicity, we assume $l=1$ 
and write $C=C_{1}$ and $r=\mathrm{rank}(C) \geq 2$. 

Take a minimal presentation complex $P$ for $G$. 
Let $\alpha _{1}, \dots , \alpha _{r}$ be simplicial paths in $P$ 
representing generators $x_{1}, \dots , x_{r}$ of $C$. 
Let $a_{i}$ denote the length of $\alpha _{i}$ for $1 \leq i \leq r$. 
Since a finite index subgroup $H$ of $G$ contains a finite index normal subgroup $\bigcap_{g \in G}gHg^{-1}$ of $G$, 
it is sufficient to consider the finite index normal subgroups of $G$. 
Suppose that $H$ is a finite index normal subgroup of $G$. 
Let $d$ denote the index of $H < G$. 
Let $\widetilde{P}$ be the covering of $P$ corresponding to $H$. 
Let $\{ C^{\prime}_{1}, \dots , C^{\prime}_{m} \}$ be 
subgroups of $H$ representing the conjugacy classes of 
$\{ gCg^{-1} \cap H \}_{g \in G}$. 
$C^{\prime}_{i}$ can be regarded as a finite index subgroup of $C$ 
by the natural inclusion $\iota _{i} \colon C^{\prime}_{i} \hookrightarrow C$. 
Since $H$ is normal in $G$, all the images of $\iota _{i}$'s coincide 
and have index $d/m$ in $C$. 
We regard $C \cong \Z ^{r}$ as a lattice in $\R ^{r}$ 
and put the 1-norm $\| \cdot \|$ in $\R ^{r}$ 
with respect to the basis $(x_{1} / a_{1}, \dots , x_{r} / a_{r})$. 

We construct a presentation complex $\widetilde{P}^{\prime}$ 
for $(H; \{ C^{\prime}_{i} \}_{1 \leq i \leq m})$ 
by attaching 2-cells to $\widetilde{P}$. 
We take a reduced basis $(y_{1}, \dots , y_{r})$ of $\iota _{i} (C^{\prime}_{i})$ 
as in Lemma \ref{lem:reducedbasis}. 
Let $\beta _{i1}, \dots , \beta _{ir}$ be paths in $\widetilde{P}$ representing 
$y_{1}, \dots , y_{r} \in \iota _{i} (C^{\prime}_{i})$ such that 
the length of $\beta _{ij}$ is $\| y_{j} \|$. 
We obtain a presentation complex $\widetilde{P}^{\prime}$ by attaching cones of 
$\beta _{ij}$'s as in the proof of Proposition \ref{prop:cone}. 
The number of the triangles of $\widetilde{P}^{\prime}$ is 
\[
d \cdot T(G) + m(\| y_{1} \| + \dots + \| y_{r} \|). 
\]
It holds that $d/m \geq \epsilon _{r} \|y_{1}\| \cdots  \|y_{r}\|$ 
by Lemma \ref{lem:reducedbasis}. 
Hence 
\[
\sT (G;C) \leq \frac{T(H; \{ C^{\prime}_{i} \}_{1 \leq i \leq m})}{d} 
\leq T(G) + 
\frac{\| y_{1} \| + \dots + \| y_{r} \|}{\epsilon _{r} \|y_{1}\| \cdots  \|y_{r}\|}
\]
Since $G$ is residually finite, there is a normal subgroup $H$ of $G$ 
such that every $\| y_{j} \|$ for $1 \leq j \leq r$ is arbitrarily large. 
We have supposed that $r \geq 2$. 
Therefore we obtain $\sT (G;C) \leq T(G)$. 

\blank

Conversely we show that $\sT (G) \leq \sT(G;C)$. 
By the same way as the above argument, 
it is sufficient to show that $\sT (G) \leq T(G;C)$. 
Take a minimal presentation complex $Q$ for $(G;C)$. 
We construct a presentation complex for $G$ by truncating 
a neighborhood of the vertex of $Q$ (Figure \ref{fig:spl-truncation}) and attaching 2-cells. 
Let $Q^{\prime}$ be the truncated complex. 
Let $\Gamma$ be the sectional graph of the truncation in $Q^{\prime}$. 
Attaching edges to $\Gamma$ if necessary, we may assume that $\Gamma$ is connected 
and the natural map from $\pi _1 (\Gamma)$ to $C$ is surjective. 
We contract vertices of $Q^{\prime}$ along edges of $\Gamma$ 
to obtain a 2-complex $Q^{\prime \prime}$. 
We obtain a bouquet $\Gamma ^{\prime}$ in $Q^{\prime \prime}$ from $\Gamma$. 
Then we have the natural surjection $p \colon \pi _1 (\Gamma ^{\prime}) \to C$. 
Attaching more edges to $\Gamma ^{\prime}$ if necessary, 
we may assume that there are edges $\gamma _{1}, \dots , \gamma _{r}$ 
such that the images of the elements 
$[\gamma _{1}], \dots , [\gamma _{r}] \in \pi _1 (\Gamma ^{\prime})$ forms a basis of $C$. 
Let $\gamma ^{\prime}_{1}, \dots , \gamma ^{\prime}_{s}$ be 
the other edges of $\Gamma ^{\prime}$. 
Write $z_{j} = p[\gamma _{j}]$ and $z^{\prime}_{k} = p[\gamma ^{\prime}_{k}]$ 
for $1 \leq j \leq r$ and $1 \leq k \leq s$. 
$z^{\prime}_{k}$ can be presented as a product of $z_{j}$'s, and 
let $b_{k}$ denote its word length. 
We obtain a presentation complex $Q^{\prime \prime \prime}$ for $G$ 
by attaching triangles to $Q^{\prime \prime}$ along $\Gamma ^{\prime}$, 
where $r(r-1)$ attached triangles correspond to the commutators 
$[z_{i}, z_{j}] = z_{i}z_{j}z_{i}^{-1}z_{j}^{-1}$ 
($1 \leq i,j \leq r$) and at most $b_{1} + \dots + b_{s} -s$ attached triangles 
correspond to the presentation of $z^{\prime}_{k}$ by $z_{j}$'s. 
Let $K$ denote the union of $\Gamma ^{\prime}$ and the attached triangles.

\fig[width=8cm]{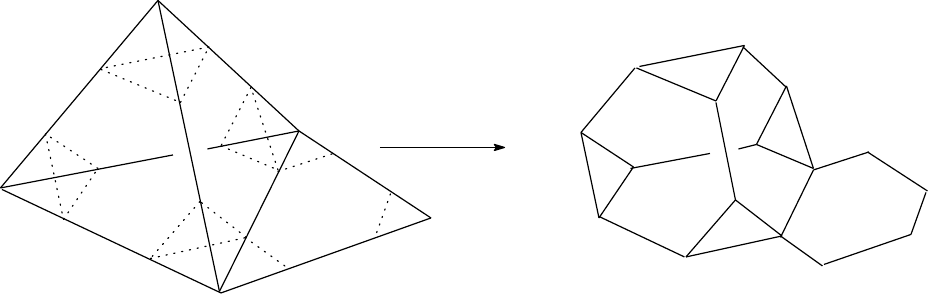}{Truncation of the presentation complex $Q$}

Suppose that $H$ is a finite index normal subgroup of $G$. 
$d$ and $\{ C^{\prime}_{1}, \dots , C^{\prime}_{m} \}$ are as above. 
Let $\widetilde{Q}$ be the covering of $Q^{\prime \prime}$ corresponding to $H$. 
Let $\widetilde{K}_{1}, \dots , \widetilde{K}_{m}$ be the components 
of covering of $K$ 
in $\widetilde{Q}$ corresponding to $\{ C^{\prime}_{1}, \dots , C^{\prime}_{m} \}$. 
Each covering $\widetilde{K}_{i} \to K$ has degree $d/m$. 
In order to construct a presentation complex $\widetilde{Q} ^{\prime}$ for $H$, 
we contract simplices of $\widetilde{K}_{i} \subset \widetilde{Q}$ 
in the following manner. 

We describe the way of contraction on the universal covering of $K$. 
We regard $\pi _{1} (K) = C$ as a lattice in $\R ^{r}$. 
Take a reduced basis of $\pi _{1} (\widetilde{K}_{i}) (< \pi _{1} (K))$. 
Let $F$ be the fundamental domain of $\pi _{1} (\widetilde{K}_{i})$
defined by this reduced basis. 
We contract simplices in the interior of $F$ 
into a point.

\fig[width=8cm]{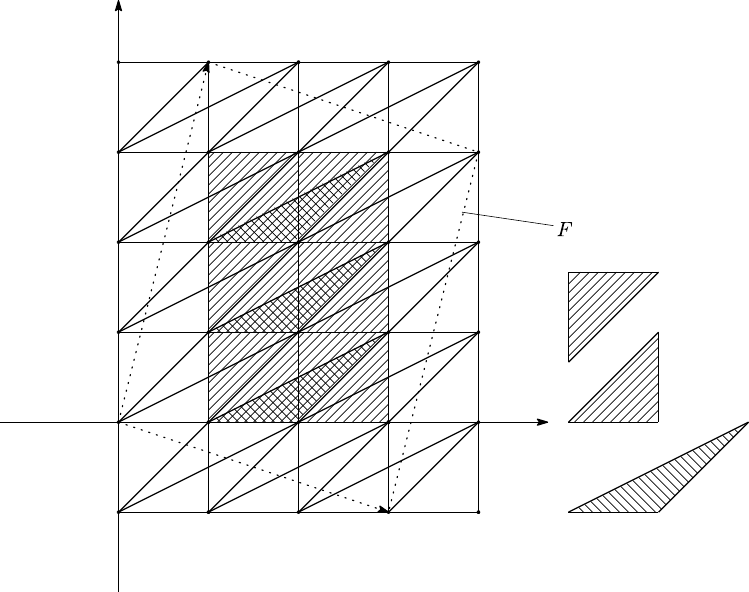}{Contraction of simplices in $F$}

We give an example in Figure \ref{fig:spl-contraction}.  
Suppose $z_{1} = (1, 0), Z_{2} = (0,1)$ and $z^{\prime}_{1} = (2, 1)$. 
The 2-complex $K$ consists of three triangles corresponding to the commutator 
$[z_{1}, z_{2}]$ and $z^{\prime}_{1} = z_{1}^{2}z_{2}$. 
Now let $((3,-1), (1,4))$ be taken as a basis of 
a lattice $\pi _{1} (\widetilde{K}_{i})$. 
Then we contract 15 triangles whose projection is in the interior of $F$.

This construction does not change the fundamental group of $\widetilde{Q}$. 
(If $r \geq 3$, this construction may change the homotopy type of $\widetilde{Q}$.) 
Thus we obtain a presentation complex $\widetilde{Q} ^{\prime}$ for $H$. 

The number of the triangles of $\widetilde{Q} ^{\prime}$ is at most 
\[
d \cdot T(G;C) + m(e+f), 
\]
where 
\begin{align*}
e &= e_{11} + \dots + e_{1r} + e_{21} + \dots + e_{2s}, \\
f &= f_{1} + f_{21} + \dots + f_{2s}, 
\end{align*}
and $e_{1j}$ and $e_{2k}$ are the numbers of the edges of $\widetilde{Q} ^{\prime}$ 
which derive from $\gamma _{j}$ and $\gamma ^{\prime}_{k}$, 
$f_{1}$ is the number of the triangles of $\widetilde{Q} ^{\prime}$ 
which derive from ones corresponding to the commutators $[x_{i}, x_{j}]$, and 
$f_{2k}$ is the number of the triangles of $\widetilde{Q} ^{\prime}$ 
which derive from ones corresponding to the presentation of $z^{\prime}_{k}$ 
by $z_{j}$'s. 
$d \cdot T(G;C) + me$ triangles of $\widetilde{Q} ^{\prime}$ derive from the hexagons 
of $Q^{\prime}$ and $mf$ triangles of $\widetilde{Q} ^{\prime}$ derive from the triangles 
of $K$. 

If the edges and triangles are not contracted by the above construction, 
they are near the boundary of $F$ in the above picture. 
Hence there exists a constant $\delta _{r} >0$ such that the followings hold: 
\begin{align*}
e_{1j} & \leq  \delta _{r} \vol (\partial F), &
 & e_{2k} \leq b_{k} \delta _{r} \vol (\partial F), \\ 
f_{1} & \leq  r(r-1) \delta _{r} \vol (\partial F), &
 & f_{2k} \leq (b_{k}-1)e_{2k}, 
\end{align*}
where $\vol (\partial F)$ is the surface area of $F$ 
with respect to the standard Euclidean metric of $\R ^{r}$. 
Therefore 
\begin{align*}
\sT (G) \leq \frac{T(H)}{d} & \leq T(G;C) + \frac{m}{d} (e+f) \\ 
& \leq T(G;C) + (r^{2} + \sum_{k=1}^{s}b_{k}^{2}) \delta _{r} \cdot 
\frac{\vol (\partial F)}{\vol (F)}. 
\end{align*}
Since $G$ is residually finite and $F$ is defined by a reduced basis, 
there is a normal subgroup $H$ of $G$ such that 
$\vol (\partial F) / \vol (F)$ is arbitrarily small. 
\end{proof}

\blank

Cooper~\cite{cooper1999volume} showed that $\vol (M) < \pi \cdot T(M)$ 
for a closed hyperbolic 3-manifold $M$. 
The isoperimetric inequality 
by Agol and Liu~\cite[Lemma 4.4]{agol2012presentation} implies that 
this inequality also holds for a cusped hyperbolic 3-manifold. 
Delzant and Potyagailo~\cite{delzant2013complexity} 
remarked that a converse inequality does not hold, 
namely, the infimum of $\vol (M) / T(M)$ 
for the hyperbolic 3-manifolds is zero. 
Indeed, hyperbolic Dehn surgery \cite[Ch. 4 and 6]{thurston1978geometry} 
gives infinitely many hyperbolic manifolds 
whose presentation length are divergent while their volumes are bounded. 
Delzant and Potyagailo used a relative presentation length 
$T(\pi _1 (M); \mathcal{E})$ to bound the volume from below, 
where $\mathcal{E}$ consists of the elementary subgroups of $\pi _1 (M)$ 
whose translation length are less than a Margulis number. 
They also showed that $\vol (M) \leq \pi \cdot T(\pi _1 (M); \mathcal{E})$ 
\cite[Theorem B]{delzant2013complexity}. 
In particular $\vol (M) \leq \pi \cdot T(M; \partial M)$. 
We use the stable presentation length to bound the volume 
instead of $T(\pi _1 (M); \mathcal{E})$. 
Cooper's inequality immediately implies that $\vol (M) \leq \pi \cdot \sT (M)$. 
A converse estimate holds for the stable presentation length. 

\begin{prop}
\label{prop:volvsspl}
The infimum of $\vol (M) / \sT (M)$ for the hyperbolic 3-manifolds is positive. 
\end{prop}

In order to show this, we mention a connection between the presentation length 
and the complexity of a 3-manifold. 
For a closed 3-manifold $M$, the $\Delta$-complexity (or Kneser complexity) 
$\sigma (M)$ is defined as the minimal number of tetrahedra 
in a triangulation of $M$. 
$\sigma (M)$ is also defined for a cusped finite volume hyperbolic 3-manifold $M$ 
by ideal triangulations. 
The complexity $c(M)$ by Matveev~\cite{matveev1990complexity} 
is the minimal number of vertices in a simple spine of $M$. 
It holds that $\sigma (M) = c(M)$ if $M$ is irreducible 
and not $S^3 , \mathbb{RP}^3$ or the lens space $L(3,1)$, 
in particular, if $M$ is a hyperbolic 3-manifold 
\cite[Theorem 5]{matveev1990complexity}. 

The stable $\Delta$-complexity $\sigma _{\infty} (M)$ and 
the stable complexity $c_{\infty} (M)$ 
of a 3-manifold $M$ are respectively defined as 
$\inf \sigma (N)/d$ and $\inf c(N)/d$ 
by taking the infimum over all the finite coverings $N$ of $M$, 
where $d$ is the degree of the covering. 
It holds that $\sigma _{\infty} (M) = c_{\infty} (M)$ 
if $M$ is a hyperbolic 3-manifold. 
$c_{\infty} (M)$ vanishes for a Seifert 3-manifold $M$, 
and $c_{\infty}$ has additivity 
for the prime decomposition and the JSJ decomposition. 

\blank

\begin{prop}
\label{prop:plvscomplexity1}
For a closed 3-manifold $M$, it holds that $T(M) \leq \sigma (M) +1$. 
\end{prop}

\begin{proof}
We take a minimal triangulation of $M$. 
Consider the 2-skeleton $P_0$ of this triangulation. 
$P_0$ has $2\sigma (M)$ triangles. 
Since a 2-complex $P$ in $M$ has a fundamental group isomorphic to $\pi _1 (M)$ 
as long as $M \setminus P$ consists of 3-balls, 
We can remove $(\sigma (M) -1)$ triangles from $P_0$ 
without changing the fundamental group. 
Therefore we obtain a presentation complex for $\pi _1 (M)$ 
with $(\sigma (M) +1)$ triangles. 
\end{proof}

\begin{prop}
\label{prop:plvscomplexity2}
For a cusped finite volume hyperbolic 3-manifold $M$, 
it holds that $T(M) \leq \sigma (M) +3$. 
\end{prop}

\begin{proof}
We take a minimal ideal triangulation of $M$. 
Consider the dual spine $P_0$ of this triangulation. 
$P_0$ has $\sigma (M)$ 2-cells, $2\sigma (M)$ edges and $\sigma (M)$ vertices. 
This $\sigma (M)$ 2-cells can be decomposed into $4\sigma (M)$ triangles. 
We contract $(\sigma (M) -1)$ vertices along edges. 
Since every edge of $P_0$ is incident on three triangles, 
we obtain a presentation complex of $\pi _1 (M)$ 
with $(\sigma (M) +3)$ triangles. 
\end{proof}

Since the fundamental group of 3-manifold is 
residually finite \cite{hempel1987residual}, 
$M$ admits arbitrarily large finite covering if $\pi _1 (M)$ is infinite. 
This implies the following corollary. 

\begin{cor}
\label{cor:splvscomplexity}
If $M$ is a closed 3-manifold or a finite volume hyperbolic 3-manifold, 
it holds that $\sT (M) \leq \sigma _{\infty}(M)$. 
\end{cor}

The stable complexity of a hyperbolic 3-manifold is bounded from above and below 
by constant multiples of its volume. 
For a finite volume hyperbolic 3-manifold $M$, 
it holds that $\vol (M) \leq V_3 \sigma (M)$, 
where $V_3$ is the volume of ideal regular tetrahedron, which is the maximum 
of the volumes of geodesic tetrahedra in the hyperbolic 3-space. 
This implies that $\vol (M) \leq V_3 \sigma _{\infty} (M)$. 
Conversely, there exists a constant $C>0$ 
such that $\sigma _{\infty} (M) \leq C \vol (M)$ holds 
for any hyperbolic manifold $M$. 
This follows from the fact by J\o rgensen and Thurston 
that a thick part of a hyperbolic 3-manifold can be decomposed 
by uniformly thick tetrahedra. 
Proofs of this fact are given by Francaviglia, Frigerio 
and Martelli~\cite[Proposition 2.5]{francaviglia2012stable} 
in the case $M$ is closed, and by Breslin \cite{breslin2009thick} and 
Kobayashi and Rieck~\cite{kobayashi2011linear} 
otherwise. 
Proposition \ref{prop:volvsspl} follows from this inequality 
and Corollary \ref{cor:splvscomplexity}. 

\blank

We conjecture an equality between the stable presentation length 
and the stable complexity. 

\begin{conj}
\label{conj}
For a finite volume hyperbolic 3-manifold $M$, it holds that 
\[\sT (M) = \frac{1}{2} \sigma _{\infty}(M).\] 
\end{conj}

We will give some examples supporting that 
$\sT (M) \leq \sigma _{\infty}(M)/2$
in Section~\ref{subsection:bianchi}. 
It holds that $T(M) \geq \sigma (M)/2$ 
if a minimal (relative) presentation complex for $\pi _1 (M)$ injects to $M$. 
This is because $M$ can be decomposed into $2T(M)$ tetrahedra. 

If Conjecture \ref{conj} holds, 
$\sT (M) = (1/2V_3) \vol (M)$ for a hyperbolic 3-manifold $M$ 
which is commensurable with the figure-eight knot complement $M_1$. 
Indeed, $\sigma _{\infty}(M_1) =2$ 
since $M_1$ can be decomposed into two ideal regular tetrahedra. 
Conjecture \ref{conj} implies a best possible refinement of Cooper's inequality 
$\vol (M) < 2V_3 \cdot T(M)$.

\section{Additivity of stable presentation length}
\label{section:additivity}

We will show additivity of the stable presentation length of 3-manifold groups 
in the same manner as the simplicial volume. 
The proofs of Theorem \ref{thm:additivityfree} and \ref{thm:additivityJSJ} 
are similar. 
Let $G$ be a finitely presentable group and let $\{ G_{i} \}$ be decomposed groups 
of $G$. 
We will construct a presentation complex for a finite index subgroup of $G$ 
by gluing finite coverings of presentation complexes for $G_{i}$. 
This implies an inequality between $\sT (G)$ and $\sum_{i} \sT (G_{i})$. 
In order to show the converse inequality, we will obtain presentation complexes 
for finite index subgroups of $G_{i}$'s by decomposing a finite covering of 
a presentation complex for $G$. 

We first show additivity for a free product. 
This holds for any finitely presentable group. 

\begin{thm}
\label{thm:additivityfree}
For finitely presentable groups $G_1$ and $G_2$, it holds that 
\[
\sT (G_1 * G_2) = \sT (G_1) + \sT (G_2). 
\]
\end{thm}

\begin{proof}
We will use additivity of presentation length for a free product 
in Corollary \ref{cor:free}. 
Write $G = G_1 * G_2$. 
We first show that $\sT (G) \leq \sT (G_1) + \sT (G_2)$. 
For $i=1,2$, let $P_i$ be presentation complexes for $G_i$. 
Take $d_i$-index subgroups $H_i$ of $G_i$. 
Let $\widetilde{P_i}$ denote the coverings of $P_i$ corresponding to $H_i$. 
Since each $\widetilde{P_i}$ has $d_i$ vertices, 
we can glue $d_2$ copies of $\widetilde{P_1}$ and $d_1$ copies of $\widetilde{P_2}$ 
along the vertices to obtain a $d_{1}d_{2}$-sheet covering $\widetilde{P}$ 
of $P_{1} \vee P_{2}$. 
The wedge sum $P_{1} \vee P_{2}$ is a presentation complex for $G$. 
Then $\pi _1 (\widetilde{P})$ is isomorphic to a free product 
$H_{1}^{*d_{2}} * H_{2}^{*d_{1}} * F_{k}$, where $F_{k}$ is a free group.  
Corollary \ref{cor:free} implies that 
$T(\pi _1 (\widetilde{P})) = d_{2} \cdot T(H_{1}) + d_{1} \cdot T(H_{2})$. 
Therefore 
\[
\sT (G) \leq \frac{T(\pi _1 (\widetilde{P}))}{d_{1}d_{2}} 
 = \frac{T(H_1)}{d_1} + \frac{T(H_2)}{d_2}. 
\]
Since we took $H_1$ and $H_2$ arbitrarily, 
we obtain $\sT (G) \leq \sT (G_1) + \sT (G_2)$. 

\blank

Conversely, we show that $\sT (G_1) + \sT (G_2) \leq \sT (G)$. 
Let $P_{i}$ be as above. 
$P = P_{1} \vee P_{2}$ is a presentation complex for $G$. 
Take a $d$-index subgroup $H$ of $G$. 
Let $\widetilde{P}$ denote the covering of $P$ corresponding to $H$. 
$\widetilde{P}$ is homotopic to 
\[
P_{11} \vee \dots \vee P_{1m} \vee P_{21} \vee \dots \vee P_{2n} 
\vee S^1 \vee \dots \vee S^1, 
\] 
where $P_{ij}$ is a covering of $P_{i}$. 
Let $d_{ij}$ be the degree of the covering $P_{ij} \to P_{i}$. 
Then $\sum_{j=1}^{m} d_{1j} = \sum_{j=1}^{n} d_{2j} = d$. 
Since 
$H= \pi_1 (\widetilde{P})$ is isomorphic to 
\[
\pi_1 (P_{11}) * \dots * \pi_1 (P_{1m}) * \pi_1 (P_{21}) * \dots * \pi_1 (P_{2n}) 
* F_{k}, 
\]
Corollary \ref{cor:free} and Proposition \ref{prop:vol} implies that 
\begin{align*}
T(H) = & T(\pi_1 (P_{11})) + \dots + T(\pi_1 (P_{1m})) 
+ T(\pi_1 (P_{21})) + \dots + T(\pi_1 (P_{2n})) \\
\geq & d_{11} \cdot \sT (\pi_1 (P_{1})) + \dots + d_{1m} \cdot \sT (\pi_1 (P_{1})) \\ 
& + d_{21} \cdot \sT (\pi_1 (P_{2})) + \dots + d_{2n} \cdot \sT (\pi_1 (P_{2})) \\
= &  d \cdot \sT (G_1) + d \cdot \sT (G_2). 
\end{align*}
Therefore $\sT (G_1) + \sT (G_2) \leq \dfrac{T(H)}{d}$. 
Since we took $H$ arbitrarily, 
we obtain $\sT (G_1) + \sT (G_2) \leq \sT (G)$. 
\end{proof}

Before we show additivity for the JSJ decomposition, 
we show that the stable presentation length for a Seifert 3-manifold vanishes. 

\begin{thm}
\label{thm:seifertvanish}
For a compact Seifert 3-manifold $M$, 
\[
\sT (M) = \sT (M; \partial M) =0. 
\]
\end{thm}

\begin{proof}
Since a Seifert 3-manifold can be regarded as an $S^1$-bundle over an 2-orbifold, 
$M$ is covered by an $S^1$-bundle over a surface. 
Hence we can assume $M$ is an $S^1$-bundle over a compact surface. 

If $M$ has boundary, $M$ is a product of $S^1$ and a surface. 
Then $M$ admits a $d$-sheeted covering homeomorphic to $M$ for any $d \leq 1$. 
This implies that $\sT (M) = \sT (M; \partial M) =0$ 
by Proposition \ref{prop:vol}. 

We consider an $S^1$-bundle over a closed surface $\Sigma _g$ of genus $g$. 
Homeomorphic class of an $S^1$-bundle over $\Sigma _g$ is determined 
by the Euler number $e$. 
Let $M(\Sigma _g, e)$ denote the $S^1$-bundle over $\Sigma _g$ of the Euler number $e$. 
Since $\pi_1 (M(S^2, e))$ is finite or isomorphic to $\Z$, 
we have $\sT (M(S^2, e)) =0$. 
Suppose $g \geq 1$. 
$\pi_1 (M(\Sigma _g, e))$ has a presentation 
\[
\langle x_1 , y_1 , \dots , x_g ,y_g , z | [x_1 , y_1] \dots [x_g , y_g] z^{e}, 
[x_i, z], [y_i ,z] \quad (1 \leq i \leq g) \rangle , 
\]
where $x_i , y_i$'s are corresponding to generators of the fundamental group 
of the base surface and $z$ is a generator of the fundamental group 
of the ordinary fiber, and $[x,y]$ denotes the commutator $xyx^{-1}y^{-1}$. 
Therefore 
\[
T(\pi_1 (M(\Sigma _g, e))) \leq 8g +|e| -2.  
\]
For any integer $d \geq 1$, $M(\Sigma _g, e)$ admits $M(\Sigma _{g^{\prime}}, de)$ 
as a $d$-sheeted covering along the base space, where $g^{\prime} = d(g-1) +1$. 
Furthermore, $M(\Sigma _{g^{\prime}}, de)$ admits $M(\Sigma _{g^{\prime}}, e)$ 
as a $d$-sheeted covering along the fiber direction. 
Thus we obtain a $d^2$-sheeted covering 
$M(\Sigma _{g^{\prime}}, e) \to M(\Sigma _g, e)$. 
Hence 
\[
\sT (\pi_1 (M(\Sigma _g, e))) \leq \frac{T(\pi_1 (M(\Sigma _{g^{\prime}}, e)))}{d^2} 
\leq \frac{8(d(g-1)+1) +|e| -2}{d^2}.  
\]
The right hand side converges to zero when $d$ increases. 
\end{proof}

Finally we show additivity for the JSJ decomposition. 

\begin{thm}
\label{thm:additivityJSJ}
Let $M$ be a compact irreducible 3-manifold 
with empty or incompressible tori boundary. 
Suppose $M = M_{1} \cup \dots \cup M_{h}$ is the JSJ decomposition. 
$M_{1}, \dots , M_{h}$ are Seifert or hyperbolic 3-manifolds 
with incompressible tori boundary. 
Then 
\[
\sT (M) = \sT (M_{1}) + \dots + \sT (M_{h}). 
\]
\end{thm}

\begin{proof}
We remark that the fundamental group of a compact 3-manifold is 
residually finite by Hempel~\cite{hempel1987residual} and the geometrization. 

We first show that $\sT (M) \leq \sT (M_{1}) + \dots + \sT (M_{h})$. 
Take $d_{i}$-sheet coverings $f_{i} \colon \widetilde{M_{i}} \to M_{i}$ 
for $1 \leq i \leq h$. 
Then there exist an integer $p$ independent of $i$ 
and coverings $g_{i} \colon N_{i} \to \widetilde{M_{i}}$ 
such that $f_{i} \circ g_{i} \colon N_{i} \to M_{i}$ is a $p$-characteristic covering, 
i.e. the restriction of the covering on each component of $\partial M_{i}$ is 
the covering corresponding to $p\Z \times p\Z < \Z \times \Z$ 
\cite[Proposition 5.7]{francaviglia2012stable}. 
We can glue copies $N_{ij}$ of $N_{i}$ along boundary to obtain a $d$-sheeted covering 
$f \colon N \to M$. Then $f^{-1} (M_{i}) = N_{11} \cup \dots \cup N_{il_{i}}$. 
Each copy $g_{ij} \colon N_{ij} \to \widetilde{M_{i}}$ of $g_{i}$ 
is a $d/l_{i}d_{i}$-sheeted covering. 
$N = \bigcup_{i,j} N_{ij}$ is the JSJ decomposition. 
Therefore we obtain 
\begin{align*}
T(\pi_1 (N);\{ \pi_1 (\partial N_{ij}) \}) 
\leq & \sum_{i,j} T(N_{ij};\partial N_{ij}) \\
\leq & \sum_{i,j} \frac{d}{l_{i}d_{i}} T(\widetilde{M_{i}}; \partial \widetilde{M_{i}}) 
= \sum_{i} \frac{d}{d_{i}} T(\widetilde{M_{i}}; \partial \widetilde{M_{i}}) 
\end{align*}
by Proposition \ref{prop:upperbound}. 
Hence 
\[
\sT (\pi_1 (M);\{ \pi_1 (\partial M_{i}) \}) 
\leq \frac{T(\pi_1 (N);\{ \pi_1 (\partial N_{ij}) \})}{d}
\leq \sum_{i} \frac{T(\widetilde{M_{i}}; \partial \widetilde{M_{i}})}{d_{i}}. 
\]
Since we took $\widetilde{M_{i}}$ arbitrarily, 
we obtain 
\[
\sT (\pi_1 (M);\{ \pi_1 (\partial M_{i}) \}) 
\leq \sum_{i} \sT(\widetilde{M_{i}}; \partial \widetilde{M_{i}}). 
\]
Furthermore, $\sT (M) = \sT (\pi_1 (M);\{ \pi_1 (\partial M_{i}) \})$ 
and $\sT (\widetilde{M_{i}}) = \sT(\widetilde{M_{i}}; \partial \widetilde{M_{i}})$ 
by Theorem \ref{thm:relabs}.

Conversely, we show that $\sT (M_{1}) + \dots + \sT (M_{h}) \leq \sT (M)$. 
Take a $d$-sheet covering $p \colon \widetilde{M} \to M$. 
Then the components $M_{ij}$ of $p^{-1} (M_{i})$ are the components of 
the JSJ decomposition of $\widetilde{M}$. 
Let $d_{ij}$ denote the degree of the covering $M_{ij} \to M_{i}$. 
Then $\sum_{j} d_{ij} = d$. 
We have 
\[
\sum_{j} T(M_{ij}; \partial M_{ij}) 
\geq \sum_{j} d_{ij} \cdot \sT (M_{i}; \partial M_{i}) 
= d \cdot \sT (M_{i}; \partial M_{i})
\]
by definition. 
Theorem \ref{thm:lowerbound} implies that 
\[
\sum_{i,j} T(M_{ij}; \partial M_{ij}) \leq T(\widetilde{M}). 
\]
Therefore it holds that 
\[
\sum_{i} \sT(M_{i}; \partial M_{i}) \leq \frac{T(\widetilde{M})}{d}. 
\]
Since we took $\widetilde{M}$ arbitrarily, 
we obtain 
\[
\sum_{i} \sT(M_{i}; \partial M_{i}) \leq \sT(M). 
\]
Furthermore, $\sT (M_{i}) = \sT(M_{i}; \partial M_{i})$ 
by Theorem \ref{thm:relabs}. 
\end{proof}

\begin{cor}
\label{cor:splvssimpvol}
There exists a constant $C>0$ such that the following holds. 
If $M$ be a closed 3-manifold, 
then 
\[
C \cdot \sT (M) \leq \| M \| \leq \frac{\pi}{V_3} \sT (M), 
\]
where $\| M \|$ is the simplicial volume of $M$ 
and $V_3$ is the volume of an ideal regular tetrahedron. 
\end{cor}

\begin{proof}
We can assume that $M$ is orientable by taking the double covering. 
Let $M = M_{1} \# \dots \# M_{n}$ be the prime decomposition. 
Each connected summand $M_{i}$ is irreducible or homeomorphic to $S^{1} \times S^{2}$. 
Let $M_{i} = M_{i1} \cup \dots \cup M_{ih_{i}}$ be the JSJ decomposition 
if $M_{i}$ is irreducible. 
The geometrization implies that each JSJ component $M_{ij}$ 
is Seifert fibered or hyperbolic. 
Let $N_{1}, \dots , N_{m}$ denote the hyperbolic components among $M_{ij}$. 
Then 
\[
\| M \| = 1/V_3 (\vol (N_{1}) + \dots + \vol (N_{m}))
\] 
by additivity and proportionality of simplicial volume \cite{gromov1982volume}. 
Now we have 
\[
\sT (M) = \sT (N_{1}) + \dots + \sT (N_{m})
\]
by Theorem \ref{thm:additivityfree}, Theorem \ref{thm:seifertvanish} and 
Theorem \ref{thm:additivityJSJ}. 
Therefore we are reduced to proving for hyperbolic 3-manifolds. 
A hyperbolic 3-manifold $M$ satisfies the above inequalities 
by Cooper's inequality and Proposition \ref{prop:volvsspl}. 
\end{proof}

\section{Examples of stable presentation length}
\label{section:example}

\subsection{Surface groups}

We calculate the explicit value of the stable presentation length 
of a surface group, 
which coincides with the simplicial volume of the surface. 

\begin{thm}
\label{thm:surface}
Let $\Sigma _{g}$ is the closed orientable surface of genus $g \geq 1$. 
Then 
\[
\sT (\pi _1 (\Sigma _{g})) = 4g-4 = -2 \chi (\Sigma _{g}).
\] 
\end{thm}

\begin{proof}
If $g=1$, $\pi _1 (\Sigma _{g}) \cong \Z \times \Z$ has a finite index proper subgroup 
isomorphic to $\Z \times \Z$. 
Then $\sT (\pi _1 (\Sigma _{g})) = 0$ by Proposition \ref{prop:vol}. 

Suppose that $g \geq 2$. 
Since there is a presentation 
\[
\pi _1 (\Sigma _{g}) = 
\langle x_1, y_1, \dots, x_g, y_g | [x_1, y_1] \cdots [x_g, y_g] \rangle , 
\]
we have $T(\pi _1 (\Sigma _{g})) \leq 4g-2$. 
In order to estimate from below, take a minimal presentation complex $P$ 
for $\pi _1 (\Sigma _{g})$. 
We put a hyperbolic metric on $\Sigma _{g}$. 
There exists a map $f \colon P \to \Sigma _{g}$ inducing an isomorphism between 
their fundamental groups. We can take $f$ which maps every 2-cell of $P$ 
to a geodesic triangle in $\Sigma _{g}$. 

We claim that $f$ is surjective. 
If $f$ is not surjective, there is a point $p$ in $\Sigma _{g} - f(P)$. 
Then $f$ induces an injection from $\pi _1 (\Sigma _{g})$ 
to $\pi _1 (\Sigma _{g} - \{ p \})$. 
Since $\pi _1 (\Sigma _{g} - \{ p \})$ is a free group 
and $\pi _1 (\Sigma _{g})$ is not a free group , we have a contradiction. 
Now $\textrm{area}(\Sigma _{g}) = (4g-4) \pi$ and the area of a geodesic triangle 
in $\Sigma _{g}$ is smaller than $\pi$. 
Hence we obtain $(4g-4) \pi < \pi \cdot T(\pi _1 (\Sigma _{g}))$. 

We finally compute $\sT (\pi _1 (\Sigma _{g}))$. 
Since $\Sigma _{d(g-1)+1}$ covers $\Sigma _{g}$ with degree $d$, 
$\sT (\pi _1 (\Sigma _{g})) \leq \frac{1}{d} T (\pi _1 (\Sigma _{d(g-1)+1})) 
\leq \frac{1}{d} (4(d(g-1)+1)-2)$. 
Hence we obtain $\sT (\pi _1 (\Sigma _{g})) \leq 4g-4$ by $d \to \infty$. 
Conversely, $4g-4 < \frac{1}{d}T(\pi _1 (\Sigma _{d(g-1)+1}))$ for any $d \geq 1$ 
implies that $4g-4 \leq \sT (\pi _1 (\Sigma _{g}))$. 
\end{proof}

\begin{thm}
\label{thm:cuspedsurface}
Let $\Sigma _{g,b}$ denote the compact orientable surface of genus $g$ 
whose boundary components are $S_{1}, \dots , S_{b}$. 
Suppose that $b>0$ and $2g-2+b>0$. 
Then 
\begin{align*}
\sT (\pi _1 (\Sigma _{g,b}); \pi _1 (S_{1}), \dots , \pi _1 (S_{b})) 
&= T(\pi _1 (\Sigma _{g,b}); \pi _1 (S_{1}), \dots , \pi _1 (S_{b})) \\
&= 4g-4+2b = -2 \chi (\Sigma _{g,b}).
\end{align*} 
\end{thm}

\begin{proof}
$\Sigma _{g,b}$ admits a hyperbolic metric with cusps $S_{1}, \dots , S_{b}$. 
An ideal triangulation of this hyperbolic surface gives a presentation complex 
for \\
$(\pi _1 (\Sigma _{g,b}); \pi _1 (S_{1}), \dots , \pi _1 (S_{b}))$, 
which consists of $4g-4+2b$ triangles. 
Therefore $T(\pi _1 (\Sigma _{g,b}); \pi _1 (S_{1}), \dots , \pi _1 (S_{b})) 
\leq 4g-4+2b$. 

In order to obtain the converse inequality, 
we put a hyperbolic metric with geodesic boundary on $\Sigma _{g,b}$. 
Take a minimal presentation complex $P$ 
for \\
$(\pi _1 (\Sigma _{g,b}); \pi _1 (S_{1}), \dots , \pi _1 (S_{b}))$. 
Let $P^{\prime}$ be the complex obtained by truncating $P$. 
There is a continuous map $f \colon P^{\prime} \to \Sigma _{g,b}$ such that 
$f$ sends the truncated section $\partial P^{\prime}$ of $P^{\prime}$ 
to the corresponding boundary components and $f$ induces an isomorphism between 
their fundamental groups. 
Then $f$ induces a map $Df \colon DP^{\prime} \to D\Sigma _{g,b}$ between 
their doubles. 
Since $Df$ induces an isomorphism between the fundamental groups, 
$Df$ is surjective by the proof of Theorem \ref{thm:surface}. 
Therefore $f$ is also surjective. 
After straightening $f$ relatively to the boundary, 
the 2-cells of $P^{\prime}$ map to right-angled hexagons, 
whose areas are equal to $\pi$. 
Then 
\[
(4g-4+2b) \pi = \textrm{area}(\Sigma _{g,b}) \leq \pi \cdot 
T(\pi _1 (\Sigma _{g,b}); \pi _1 (S_{1}), \dots , \pi _1 (S_{b})). 
\]

Now we have 
$T(\pi _1 (\Sigma _{g,b}); \pi _1 (S_{1}), \dots , \pi _1 (S_{b})) 
= 4g-4+2b$. 
Since these values are already volume-like, their stable presentation lengths 
coincide with their presentation lengths. 
\end{proof}

\subsection{Bianchi groups}
\label{subsection:bianchi}

We consider the stable presentation lengths of Bianchi groups 
$\PSL (2, \mathcal{O}_{d})$, 
where $\mathcal{O}_{d}$ is the ring of integers 
in the imaginary quadratic field $\Q (\sqrt{-d})$, namely,  
\[
\mathcal{O}_{d} = 
\begin{cases}
\Z [ \frac{1+\sqrt{-d}}{2} ] & \text{if} -d \equiv 1 \mod 4 \\
\Z [\sqrt{-d}] & \text{if} -d \equiv 2, 3 \mod 4. 
\end{cases} 
\]
It is known that the fundamental group of 
every finite volume cusped arithmetic hyperbolic 3-manifold 
is commensurable with a Bianchi group 
\cite[Proposition 4.1]{neumann1992arithmetic}. 
Hatcher~\cite{hatcher1983hyperbolic} showed that 
some Bianchi groups preserve tesselations of $\mathbb{H}^{3}$ 
by ideal uniform polyhedra. 
Consequently, 
hyperbolic 3-manifolds obtained from certain ideal uniform polyhedra 
are arithmetic.

We will give upper bounds of stable presentation lengths 
of some arithmetic link components 
by constructing explicit presentations of their fundamental groups. 
We consider the complements of links in $T^2 \times [0,1]$. 
Since the complement of the Hopf link is homeomorphic 
to $T^2 \times [0,1]$, 
the complement of a link in $T^2 \times [0,1]$ 
is homeomorphic to the complement of a link in $S^{3}$. 
As we will see later, 
there are infinitely many links in $T^2 \times [0,1]$ 
whose complements admit hyperbolic structures. 

For a hyperbolic link $K$ in $T^2 \times [0,1]$, 
the two boundary components of $T^2 \times [0,1]$ 
and the components of $K$ correspond to the cusps. 
We will take finite coverings of $T^2 \times [0,1] \setminus K$ 
induced by ones of $T^2 \times [0,1]$. 
These coverings are the complement of links $K_{m,n}$ 
in $T^2 \times [0,1]$. 
We can obtain Wirtinger presentations 
of $\pi_{1} (T^2 \times [0,1] \setminus K_{m,n})$ from their diagrams 
analogously to the ones for links in $S^{3}$. 
We will need additional generators 
to obtain presentations of shorter lengths. 

We can obtain an explicit presentation complex 
from an ideal triangulation of $T^2 \times [0,1] \setminus K$. 
For instance, the 2-skeleton of an ideal triangulation 
is a presentation complex for $\pi_{1} (T^2 \times [0,1] \setminus K)$ 
relatively to the fundamental groups of the cusps. 
If there is an alternating diagram of $K$, 
We can systematically obtain an ideal polyhedral decomposition 
of $T^2 \times [0,1] \setminus K$ from the diagram 
analogously to the ideal decomposition of alternating links 
due to Menasco~\cite{menasco1983polyhedral}. 
This argument will be applied 
in Section~\ref{subsubsection:d3} and \ref{subsubsection:d1}. 
The ideal decomposition in Section~\ref{subsubsection:d1} 
was explained in detail 
by Champanerkar, Kofman and Purcell~\cite[Section 3]{champanerkar2016geometrically}. 

We can also obtain a small presentation complex 
from an ideal even triangulation. 
An (ideal) triangulation $\mathcal{T}$ of a 3-manifold $M$ determines 
the projection $p$ from the tetrahedra 
to the end-compactification of $M$. 
Following Rubinstein and Tillmann~\cite{rubinstein2015even}, 
$\mathcal{T}$ is an \textit{even} triangulation 
if the preimage $p^{-1}(\tau)$ of each edge $\tau$ in $\mathcal{T}$ is 
even number of edges. 
A \textit{vertex coloring} of $\mathcal{T}$ 
is a map from the vertices in $\mathcal{T}$ to $\{ 0,1,2,3 \}$ 
such that its restriction to the vertices of each tetrahedron 
is bijective. 
Although the universal covering of an even triangulation 
admits a vertex coloring, 
the deck transformations may not preserve the coloring. 
This difference determines a monodromy homomorphism from $\pi_{1}(M)$ 
to the symmetric group $\mathrm{Sym}(4)$ on $\{ 0,1,2,3 \}$, 
which is called a \textit{symmetric representation} 
for an even triangulation in \cite{rubinstein2015even}.

\begin{lem}
\label{lem:even}
Suppose that a finite volume hyperbolic 3-manifold $M$ 
admits an ideal even triangulation $\mathcal{T}$ with $n$ tetrahedra. 
Then $\sT (M) \leq n/2$. 
\end{lem}

\begin{proof}
Let $\rho \colon \pi_{1}(M) \to \mathrm{Sym}(4)$ 
be a symmetric representation. 
The manifold $M$ has a finite covering $M^{\prime}$ 
which corresponds to $\ker (\rho)$. 
Then the lifted triangulation $\mathcal{T}^{\prime}$ 
of $\mathcal{T}$ for $M^{\prime}$ admits a vertex coloring. 
We take the $\deg(M^{\prime} \to M)n/2$ triangles in $\mathcal{T}^{\prime}$ 
which do not contain a vertex of color 0. 
The union of these triangles 
is a presentation complex for $\pi_{1} (M^{\prime})$ 
relative to the fundamental groups of the cusps of colors $\{ 1,2,3 \}$. 
Therefore Theorem~\ref{thm:relabs} implies that 
$\sT (M) = \sT (M^{\prime})/\deg(M^{\prime} \to M) \leq n/2$. 
\end{proof}

We will give explicit examples for the above construction 
in Section \ref{subsubsection:d3} and \ref{subsubsection:d1}.

\subsubsection{$d=3$ (Figure-eight knot complement)}
\label{subsubsection:d3}

The figure-eight knot complement $M_{1}$ is obtained 
from two ideal regular tetrahedra. 
Hence $\vol (M_{1}) = 2 V_{3} = 2.0298...$ 
and $\sigma (M_{1}) = \sigma_{\infty}(M_{1}) = 2$.  
It is known that $\pi_{1} (M_{1})$ is isomorphic to 
an index 12 subgroup of $\PSL (2, \mathcal{O}_{3})$. 
The index follows from 
Humbert's formula~\cite[Theorem 7.4.1]{thurston1978geometry}
for $\vol (\mathbb{H}^{3} / \PSL (2, \mathcal{O}_{d}))$. 

For a general case, 
suppose that a hyperbolic 3-manifold $M$ is obtained 
from ideal regular tetrahedra. 
Then the action of $\pi_{1} (M)$ on $\mathbb{H}^{3}$ 
preserves the tessellation by ideal regular tetrahedra of $\mathbb{H}^{3}$. 
Since the symmetry group of this tessellation is commensurable with 
$\PSL (2, \mathcal{O}_{3})$, 
the groups $\pi_{1} (M), \pi_{1} (M_{1})$ and $\PSL(2, \mathcal{O}_{3})$ 
are commensurable 
as shown in \cite[Section 3, Example 2]{hatcher1983hyperbolic}.

\begin{prop}
\label{prop:figure8}
\[
\sT (M_{1}) \leq 1. 
\]
\end{prop}

\begin{proof}
Let $M_{1,1}$ denote the complement of 
the link in Figure~\ref{fig:spl-d3link}. 
By taking the Hopf sublink consisting of 
the two components shown using thin lines, 
we regard $M_{1,1}$ as the complement of 
a link in $T^2 \times [0,1]$, 
which is constructed by gluing of the piece in Figure~\ref{fig:spl-link1} 
along the faces of top and bottom, left and right. 

The manifold $M_{1,1}$ can be decomposed 
into four ideal regular hexagonal pyramids 
as shown in Figure~\ref{fig:spl-gluing1}. 
Since the union of two ideal regular hexagonal pyramids can be decomposed 
into six ideal regular tetrahedra, 
the manifold $M_{1,1}$ is obtained from 12 ideal regular tetrahedra. 
Since the manifolds $M_{1,1}$ and $M_{1}$ are commensurable, 
we have $\sT(M_{1,1})/\sT(M_1) = \vol(M_{1,1})/\vol(M_1) =6$.

Let $M_{m,n}$ denote the $mn$-sheeted covering of $M_{1,1}$ 
which is the $m$-sheeted covering along $s$ and the $n$-sheeted covering along $t$ 
as in Figure~\ref{fig:spl-gen1}. 
The diagram gives a Wirtinger presentation of $\pi_{1} (M_{m,n})$. 
We put a base point of $\pi_{1} (M_{m,n})$ at the upper left front point. 
The generators are 
\[
x_{ij}, y_{ij}, z_{ij}, w_{ij}, 
x_{m+1,j}, y_{m+1,j}, x_{i,n+1}, z_{i,n+1}, x_{m+1,n+1}, s, t, 
\]
and the relators are 
\begin{align*}
z_{ij} &= y_{ij}x_{ij}y_{ij}^{-1}, 
& w_{ij} &= z_{ij}y_{ij}z_{ij}^{-1}, \\
x_{i+1,j+1} &= w_{ij}^{-1}z_{i,j+1}w_{ij}, 
& y_{i+1,j} &= x_{i+1,j+1}^{-1}w_{ij}x_{i+1,j+1}, \\
x_{m+1,j} &= sx_{1,j}s^{-1}, & x_{m+1,n+1} &= sx_{1,n+1}s^{-1}, 
\quad y_{m+1,j} = sy_{1,j}s^{-1}, \\ 
x_{i,n+1} &= tx_{i,1}t^{-1}, & x_{m+1,n+1} &= tx_{m+1,1}t^{-1}, 
\quad z_{i,n+1} = tz_{i,1}t^{-1}, \\
st &= ts, 
\end{align*}
for $1 \leq i \leq m, 1 \leq j \leq n$. 
The generators $x_{ij}, y_{ij}, z_{ij}, w_{ij}$ 
correspond to the arcs in the diagram, 
Some relators correspond to the crossings of the link, 
and the others come from the actions of $s$ and $t$.

We add generators $a_{ij}$ and $b_{ij}$ for smaller presentation length. 
Thus we obtain an explicit presentation of $\pi_{1} (M_{m,n})$: 
the generators are 
\[
x_{ij}, y_{ij}, z_{ij}, w_{ij}, a_{ij}, b_{ij}, 
x_{m+1,j}, y_{m+1,j}, x_{i,n+1}, z_{i,n+1}, x_{m+1,n+1}, s, t, 
\]
and the relators are 
\begin{align*}
a_{ij} &= y_{ij}x_{ij}, & a_{ij} &= z_{ij}y_{ij}, & a_{ij} &= w_{ij}z_{ij}, \\ 
b_{ij} &= z_{i,j+1}w_{ij}, & b_{ij} &= w_{ij}x_{i+1,j+1}, & b_{ij} &= x_{i+1,j+1}y_{i+1,j}, \\ 
x_{m+1,j} &= sx_{1,j}s^{-1}, & x_{m+1,n+1} &= sx_{1,n+1}s^{-1}, 
& y_{m+1,j} &= sy_{1,j}s^{-1}, \\ 
x_{i,n+1} &= tx_{i,1}t^{-1}, &x_{m+1,n+1} &= tx_{m+1,1}t^{-1}, 
& z_{i,n+1} &= tz_{i,1}t^{-1}, \quad st=ts, 
\end{align*}
for $1 \leq i \leq m, 1 \leq j \leq n$. 
Therefore 
\[
\sT(M_{1,1}) \leq \inf_{m,n} \frac{T(M_{m,n})}{mn} 
\leq \inf_{m,n} \frac{6mn+4m+4n+6}{mn} = 6. 
\]
\end{proof}

\begin{rem}
In fact, Proposition~\ref{prop:figure8} follows from Lemma~\ref{lem:even}, 
since a triangulation made of ideal regular tetrahedra is even. 
For some possible use, however, 
we gave an explicit presentation of $\pi_{1} (M_{m,n})$. 

It is also possible to construct an explicit relative presentation complex 
as in the proof of Lemma~\ref{lem:even}. 
The manifold $M_{1,1}$ has four cusps $S_{0}, S_{1}, S_{2}, S_{3}$, 
where $S_{0}$ and $S_{1}$ are the boundary component of 
$T^2 \times [0,1]$. 
We construct a fundamental domain $X$ of $M_{1,1}$ 
as a union of 12 ideal regular tetrahedra such that 
$S_{0}$ corresponds to a single vertex $v$ of $X$ 
(Figure~\ref{fig:spl-domain}). 
Then we obtain a presentation complex for 
$(\pi_{1}(M_{1,1}); \pi_{1}(S_{1}), \pi_{1}(S_{2}), \pi_{1}(S_{3}))$ 
from the triangles in $\partial X$ which do not contain $v$. 
Hence 
$T(\pi_{1}(M_{1,1}); \pi_{1}(S_{1}), \pi_{1}(S_{2}), \pi_{1}(S_{3})) \leq 6$. 
Theorem~\ref{thm:relabs} implies that 
\[
\sT (M_{1,1}) = 
\sT (\pi_{1}(M_{1,1}); \pi_{1}(S_{1}), \pi_{1}(S_{2}), \pi_{1}(S_{3}))  
\leq 6. 
\]

\end{rem}

\fig[width=5cm]{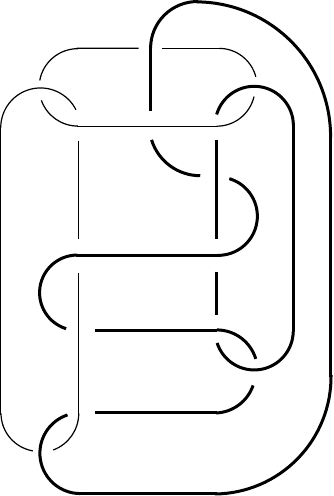}{A link whose complement is $M_{1,1}$}
\fig[width=10cm]{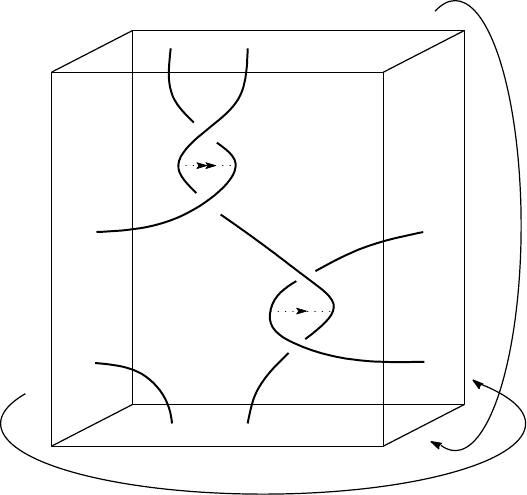}{$M_{1,1}$ as the complement of a link in $T^2 \times [0,1]$}
\fig[width=11cm]{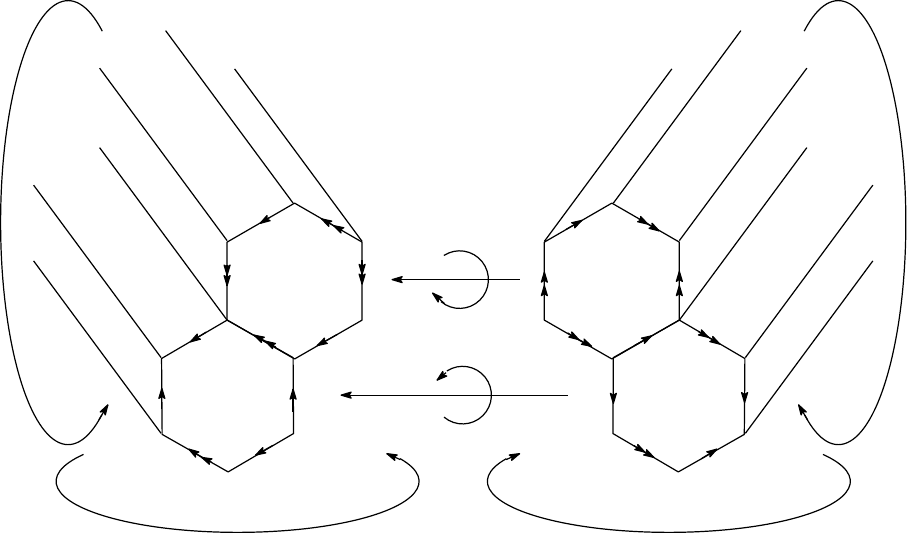}{A decomposition of $M_{1,1}$}
\fig[width=8cm]{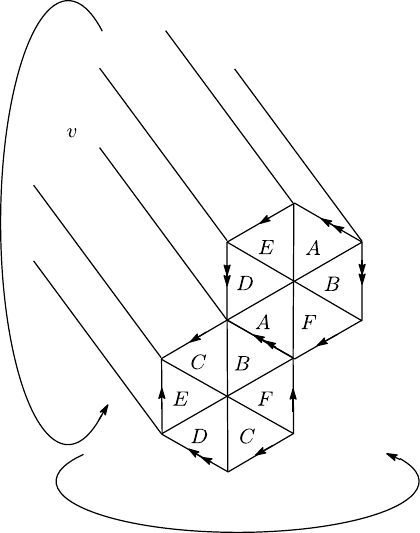}{A fundamental domain $X$ of $M_{1,1}$}
\fig[width=12cm]{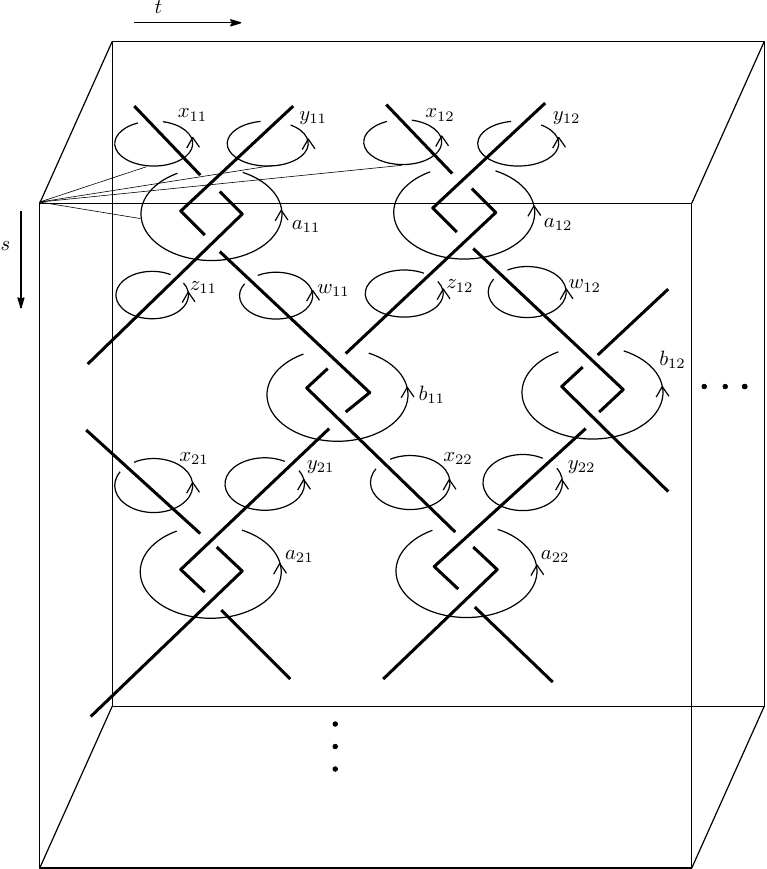}{Generators of $\pi_{1}(M_{m,n})$}

\subsubsection{$d=1$ (Whitehead link complement)}
\label{subsubsection:d1}

The Whitehead link complement $M_{2}$ is obtained 
from one ideal regular octahedron. 
Since $\vol (M_{2}) = 3.6638...$, 
we have $\sigma (M_{2}) = 4$ and $3.6 < \sigma_{\infty}(M_{2}) \leq 4$. 
It is unknown whether $\sigma_{\infty}(M_{2}) = 4$ or not. 
It is known that 
$\pi_{1} (M_{2})$ is an index 12 subgroup of $\PSL (2, \mathcal{O}_{1})$. 
If a hyperbolic manifold $M$ is obtained from ideal regular octahedra, 
the group $\pi_{1}(M)$ is commensurable with $\PSL (2, \mathcal{O}_{1})$ 
as shown in \cite[Section 3, Example 3]{hatcher1983hyperbolic}.

\begin{prop}
\label{prop:Whitehead}
\[
\sT(M_{2}) \leq 2. 
\]
\end{prop}

\begin{proof}
Since there is an even triangulation of $M_{2}$ 
with four ideal tetrahedra \cite[Example 4]{rubinstein2015even}, 
Lemma~\ref{lem:even} implies the assertion. 
As with the above proposition, however, 
we consider links in $T^2 \times [0,1]$. 
Let $M_{2}^{\prime}$ denote the complement of the link 
in Figure~\ref{fig:spl-d1link}. 
We regard $M_{2}^{\prime}$ as the complement of 
a link in $T^2 \times [0,1]$ 
as shown in Figure~\ref{fig:spl-gen2}. 
 
The manifold $M_{2}^{\prime}$ can be decomposed into 
four ideal regular square pyramids 
as shown in Figure~\ref{fig:spl-gluing2}. 
Since a union of two ideal regular square pyramids 
is an ideal regular octahedron, 
the manifold $M_{2}^{\prime}$ is obtained from two ideal regular octahedra. 
Since $M_{2}^{\prime}$ and $M_{2}$ are commensurable, 
we have 
$\sT(M_{2}^{\prime})/\sT(M_{2}) = \vol(M_{2}^{\prime})/\vol(M_{2}) =2$. 

We obtain a Wirtinger presentation of $\pi_{1} (M_{2}^{\prime})$: 
the generators are 
\[
x_{11}, x_{21}, x_{22}, y_{11}, y_{12}, y_{21}, s, t, 
\]
and the relators are 
\begin{align*}
x_{22} &= y_{11}x_{11}y_{11}^{-1}, & y_{21} &= x_{22}^{-1}y_{12}x_{22}, \\
x_{21} &= sx_{11}s^{-1}, & y_{21} &= sy_{11}s^{-1}, \\
y_{12} &= ty_{11}t^{-1}, & x_{22} &= tx_{21}t^{-1}, & st=ts. 
\end{align*}

After we take large coverings along $T^2 \times [0,1]$ 
as with Proposition~\ref{prop:figure8}, 
the relators which does not contain $s$ or $t$ contribute 
an estimate of the stable presentation length. 
Therefore $\sT(M_2^{\prime}) \leq 4$. 
\end{proof}

We remark that 
it is possible to prove that $\sT(M_2^{\prime}) \leq 4$ by constructing 
a relative presentation complex as with Proposition~\ref{prop:figure8}. 

\fig[width=6cm]{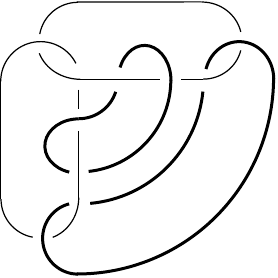}{A link whose complement is $M_{2}^{\prime}$}
\fig[width=8cm]{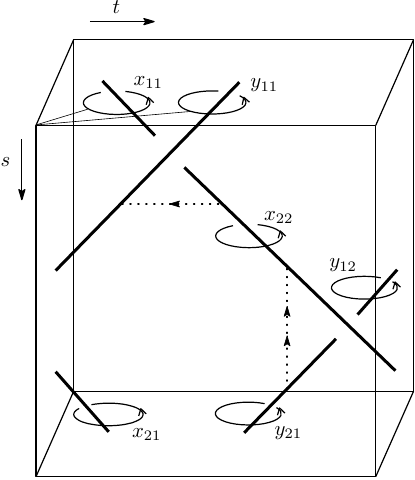}{Generators of $\pi_{1}(M_{2}^{\prime})$}
\fig[width=10cm]{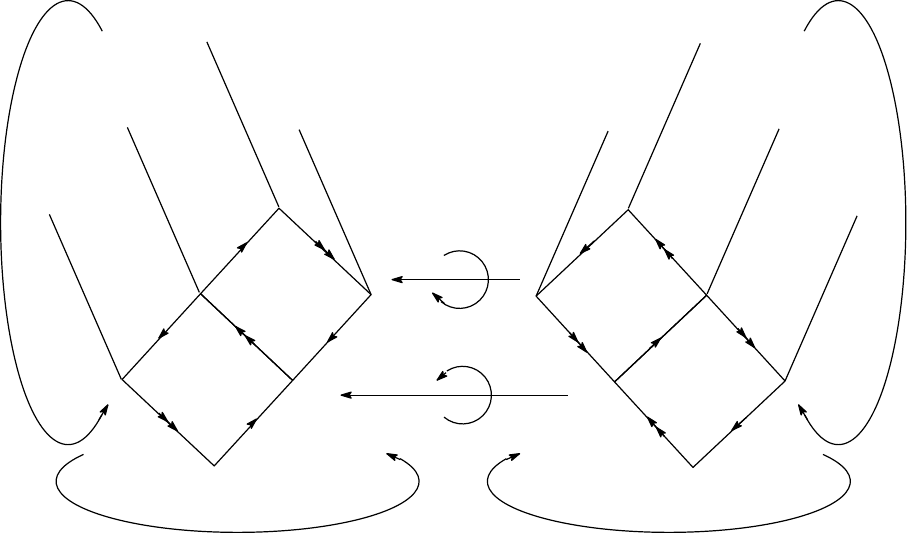}{A decomposition of $M_{2}^{\prime}$}

\subsubsection{$d=7$ (Magic manifold)}
\label{subsubsection:d7}

Let $M_{3}$ denote the complement of the alternating 3-chain link 
in Figure~\ref{fig:spl-d7link}. 
Gordon and Wu~\cite{gordon1999toroidal} 
called $M_{3}$ the magic manifold 
for the reason that 
it gives various interesting examples of Dehn fillings. 
Martelli and Petronio~\cite{martelli2006dehn} 
classified the non-hyperbolic Dehn fillings of $M_{3}$. 
The manifold $M_{3}$ is obtained from two ideal uniform triangular prism. 
Since $\vol (M_{3}) = 5.3334...$, 
we have $\sigma (M_{3}) = 6$ and $ 5.2 < \sigma_{\infty}(M_{3}) \leq 6$. 
The group $\pi _{1} (M_{3})$ is 
an index 6 subgroup of $\PSL (2, \mathcal{O}_{7})$ 
as shown in \cite[Ch.6, Example 6.8.2]{thurston1978geometry}. 

\begin{prop}
\label{prop:magic}
\[
\sT(M_{3}) \leq 3. 
\]
\end{prop}

\begin{proof}
The manifold $M_{3}$ is homeomorphic to 
the complement of a link in $T^2 \times [0,1]$ 
as shown in Figure \ref{fig:spl-gen3}. 
We obtain an explicit presentation of $\pi_{1} (M_{3})$: 
the generators are 
\[
x_{11}, x_{21}, y_{11}, y_{12}, a, s, t, 
\]
and the relators are 
\begin{align*}
a &= y_{12}x_{11}, & a &= x_{11}x_{21}, & a &= x_{21}y_{11}, \\
x_{21} &= sx_{11}s^{-1}, & y_{12} &= ty_{11}t^{-1}, & st &= ts. 
\end{align*}

After we take large coverings along $T^2 \times [0,1]$ 
as with the above propositions, 
the relators which does not contain $s$ or $t$ contribute 
an estimate of the stable presentation length. 
Therefore $\sT(M_{3}) \leq 3$. 
\end{proof}

\fig[width=6cm]{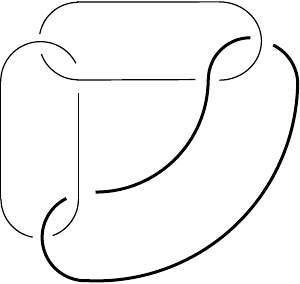}{The alternating 3-chain link}
\fig[width=8cm]{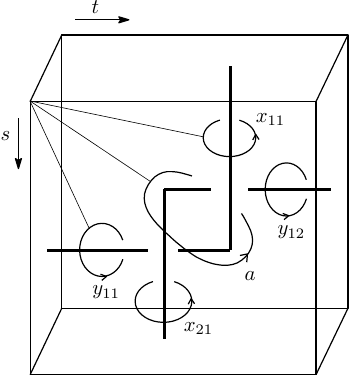}{Generators of $\pi_{1}(M_{3})$}

\subsubsection{$d=2$}
\label{subsubsection:d2}

If a hyperbolic manifold $M$ is obtained from ideal uniform cuboctahedra, 
the group $\pi_{1}(M)$ is commensurable with $\PSL (2, \mathcal{O}_{2})$ 
as shown in \cite{hatcher1983hyperbolic}. 
Let $M_{4}$ denote the complement of the link 
in Figure~\ref{fig:spl-d2link}. 
This link was shown in \cite[Figure 1b]{baker2014principal}. 
The manifold $M_{4}$ is obtained from four ideal uniform cuboctahedra. 
Let $N_{4}$ denote the double of an ideal uniform cuboctahedron 
along the ideal squares. 
Then $M_{4}$ is the double of $N_{4}$ along the 3-punctured spheres. 
We regard five components of the link in Figure~\ref{fig:spl-d2link} 
as horizontal and the seven others as vertical. 
If we cut $M_{4}$ along six horizontal 4-punctured spheres 
and eight vertical 3-punctured spheres, 
then we obtain four ideal uniform cuboctahedra. 
Since $\vol (M_{4}) = 48.1843...$ and a cuboctahedron can be decomposed 
into 14 tetrahedra compatible to the decomposition of $M_4$, 
we have $48 \leq \sigma (M_{4}) \leq 56$ 
and $ 47.4 < \sigma_{\infty}(M_{4}) \leq 56$. 
Although this ideal triangulation of $M_{4}$ is not even, 
the argument in the proof of Lemma~\ref{lem:even} can be applied. 

\begin{prop}
\label{prop:d2}
\[
\sT(M_{4}) \leq 28. 
\]
\end{prop}

\begin{proof}
Let $P$ denote the 2-skeleton of the decomposition of $M_{4}$ 
into four ideal uniform cuboctahedra. 
Let $Q$ denote the subcomplex of $P$ 
consisting of the cells which do not contain a fixed vertex $v$. 
Then the 2-cells of $Q$ consist of 12 triangles and 8 squares. 
By decomposing each square of $Q$ into two triangles, 
we obtain a presentation complex $Q^{\prime}$ for $\pi_{1} (M_{4})$ 
relative to the fundamental groups of the other cusps than $v$. 
Since the 2-cells of $Q$ consist of 28 triangles, 
we have $\sT(M_{4}) \leq 28$ by Theorem~\ref{thm:relabs}. 
\end{proof}

\fig[width=4cm]{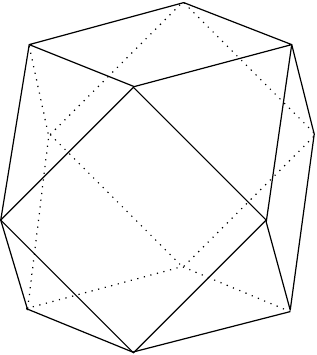}{A cuboctahedron}
\fig[width=10cm]{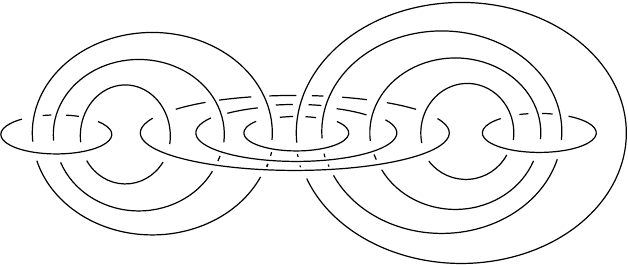}{A link whose complement is $M_{4}$}

\bibliography{ref-spl}

\begin{thebibliography}{10}

\bibitem{agol2013virtual}
I.~Agol, D.~Groves, and J.~Manning.
\newblock The virtual {H}aken conjecture.
\newblock {\em Doc. Math.}, 18:1045--1087, 2013.

\bibitem{agol2012presentation}
I.~Agol and Y.~Liu.
\newblock Presentation length and {S}imon's conjecture.
\newblock {\em J. Amer. Math. Soc.}, 25(1):151--187, 2012.

\bibitem{baker2014principal}
M.~D. Baker and A.~W. Reid.
\newblock Principal congruence link complements.
\newblock {\em Ann. Fac. Sci. Toulouse Math. (6)}, 23:1063--1092, 2014.

\bibitem{bergeron2013asymptotic}
N.~Bergeron and A.~Venkatesh.
\newblock The asymptotic growth of torsion homology for arithmetic groups.
\newblock {\em J. Inst. Math. Jussieu}, 12(2):391--447, 2013.

\bibitem{boileau1984heegaard}
M.~Boileau and H.~Zieschang.
\newblock Heegaard genus of closed orientable {S}eifert 3-manifolds.
\newblock {\em Invent. Math.}, 76(3):455--468, 1984.

\bibitem{breslin2009thick}
W.~Breslin.
\newblock Thick triangulations of hyperbolic $n$-manifolds.
\newblock {\em Pacific J. Math.}, 241(2):215--225, 2009.

\bibitem{cassels1971introduction}
J.~W.~S. Cassels.
\newblock {\em An Introduction to ther Geometry of Numbers}.
\newblock Springer-Verlag, 1971.

\bibitem{champanerkar2016geometrically}
A.~Champanerkar, I.~Kofman, and J.~S. Purcell.
\newblock Geometrically and diagrammatically maximal knots.
\newblock {\em J. Lond. Math. Soc.}, 94(3):883--908, 2016.

\bibitem{cooper1999volume}
D.~Cooper.
\newblock The volume of a closed hyperbolic 3-manifold is bounded by $\pi$
  times the length of any presentation of its fundamental group.
\newblock {\em Proc. Amer. Math. Soc.}, 127(3):941--942, 1999.

\bibitem{delzant1996decomposition}
T.~Delzant.
\newblock D{\'e}composition d'un groupe en produit libre ou somme
  amalgam{\'e}e.
\newblock {\em J. Reine Angew. Math.}, 470:153--180, 1996.

\bibitem{delzant2013complexity}
T.~Delzant and L.~Potyagailo.
\newblock On the complexity and volume of hyperbolic 3-manifolds.
\newblock {\em Israel J. Math.}, 193(1):209--232, 2013.

\bibitem{francaviglia2012stable}
S.~Francaviglia, R.~Frigerio, and B.~Martelli.
\newblock Stable complexity and simplicial volume of manifolds.
\newblock {\em J. Topol.}, 5(4):977--1010, 2012.

\bibitem{frigerio2016integral}
R.~Frigerio, C.~L\"{o}h, C.~Pagliantini, and R.~Sauer.
\newblock Integral foliated simplicial volume of aspherical manifolds.
\newblock {\em Israel J. Math.}, 216(2):707--751, 2016.

\bibitem{gordon1999toroidal}
C.~Gordon and Y.-Q. Wu.
\newblock Toroidal and annular {D}ehn fillings.
\newblock {\em Proc. Lond. Math. Soc.}, 78(3):662--700, 1999.

\bibitem{gromov1982volume}
M.~Gromov.
\newblock Volume and bounded cohomology.
\newblock {\em Inst. Hautes \'{E}tudes Sci. Publ. Math.}, 56:5--99, 1982.

\bibitem{haefliger1991complexes}
A.~Haefliger.
\newblock Complexes of groups and orbihedra.
\newblock In {\em Group Theory from a Geometrical Viewpoint}, pages 504--540.
  ICTP Trieste, World Scientific, 1991.

\bibitem{hatcher1983hyperbolic}
A.~Hatcher.
\newblock Hyperbolic structures of arithmetic type on some link complements.
\newblock {\em J. Lond. Math. Soc.}, 2(2):345--355, 1983.

\bibitem{hempel1987residual}
J.~Hempel.
\newblock Residual finiteness for 3-manifolds.
\newblock In {\em Combinatorial Group Theory and Topology}, volume 111 of {\em
  Ann. of Math. Studies}, pages 379--396. Princeton Univ. Press, 1987.

\bibitem{kahn2015good}
J.~Kahn and V.~Markovic.
\newblock The good pants homology and the {E}hrenpreis conjecture.
\newblock {\em Ann. of Math.}, 182(1):1--72, 2015.

\bibitem{kobayashi2011linear}
T.~Kobayashi and Y.~Rieck.
\newblock A linear bound on the tetrahedral number of manifolds of bounded
  volume (after {J}\o rgensen and {T}hurston).
\newblock In {\em Topology and Geometry in Dimension Three: Triangulations,
  Invariants, and Geometric Structures}, volume 560 of {\em Contemp. Math.},
  pages 27--42. Amer. Math. Soc., 2011.

\bibitem{lackenby2005expanders}
M.~Lackenby.
\newblock Expanders, rank and graphs of groups.
\newblock {\em Israel J. Math.}, 146(1):357--370, 2005.

\bibitem{lackenby2006heegaard}
M.~Lackenby.
\newblock Heegaard splittings, the virtually {H}aken conjecture and property
  ($\tau$).
\newblock {\em Invent. Math.}, 164(2):317--359, 2006.

\bibitem{le2014growth}
T.~Le.
\newblock Growth of homology torsion in finite coverings and hyperbolic volume.
\newblock {\em arXiv preprint}, 2014.
\newblock \url{http://arxiv.org/abs/1412.7758}.

\bibitem{lenstra1982factoring}
A.~K. Lenstra, H.~W. Lenstra, and L.~Lov\'{a}sz.
\newblock Factoring polynomials with rational coefficients.
\newblock {\em Math. Ann.}, 261(4):515--534, 1982.

\bibitem{li2013rank}
T.~Li.
\newblock Rank and genus of 3-manifolds.
\newblock {\em J. Amer. Math. Soc.}, 26(3):777--829, 2013.

\bibitem{loh2009simplicial}
C.~L{\"o}h and R.~Sauer.
\newblock Simplicial volume of {H}ilbert modular varieties.
\newblock {\em Comment. Math. Helv.}, 84(3):457--470, 2009.

\bibitem{luck2013approximating}
W.~L\"{u}ck.
\newblock Approximating {$L^2$}-invariants and homology growth.
\newblock {\em Geom. Funct. Anal.}, 23(2):622--663, 2013.

\bibitem{martelli2006dehn}
B.~Martelli and C.~Petronio.
\newblock Dehn filling of the ``magic'' 3-manifold.
\newblock {\em Comm. Anal. Geom.}, 14(5):969--1026, 2006.

\bibitem{matveev1990complexity}
S.~V. Matveev.
\newblock Complexity theory of three-dimensional manifolds.
\newblock {\em Acta Appl. Math.}, 19(2):101--130, 1990.

\bibitem{menasco1983polyhedral}
W.~Menasco.
\newblock Polyhedral representation of link complements.
\newblock {\em Contemp. Math.}, 20:305--325, 1983.

\bibitem{milnor1977characteristic}
J.~Milnor and W.~P. Thurston.
\newblock Characteristic numbers of 3-manifolds.
\newblock {\em Enseign. Math.}, 23:249--254, 1977.

\bibitem{mostow1973strong}
G.~D. Mostow.
\newblock {\em Strong rigidity of locally symmetric spaces}.
\newblock Number~78 in Ann. of Math. Studies. Princeton University Press, 1973.

\bibitem{neumann1992arithmetic}
W.~D. Neumann and A.~W. Reid.
\newblock Arithmetic of hyperbolic manifolds.
\newblock {\em Topology}, 90:273--310, 1992.

\bibitem{perelman2002entropy}
G.~Perelman.
\newblock The entropy formula for the {R}icci flow and its geometric
  applications.
\newblock {\em arXiv preprint}, 2002.
\newblock \url{http://arxiv.org/abs/math/0211159}.

\bibitem{perelman2003ricci}
G.~Perelman.
\newblock {R}icci flow with surgery on three-manifolds.
\newblock {\em arXiv preprint}, 2003.
\newblock \url{http://arxiv.org/abs/math/0303109}.

\bibitem{pervova2008complexity}
E.~Pervova and C.~Petronio.
\newblock Complexity and {$T$}-invariant of {A}belian and {M}ilnor groups, and
  complexity of 3-manifolds.
\newblock {\em Math. Nachr.}, 281(8):1182--1195, 2008.

\bibitem{prasad1973strong}
G.~Prasad.
\newblock Strong rigidity of $\mathbb{Q}$-rank 1 lattices.
\newblock {\em Invent. Math.}, 21(4):255--286, 1973.

\bibitem{reznikov1996volumes}
A.~Reznikov.
\newblock Volumes of discrete groups and topological complexity of homology
  spheres.
\newblock {\em Math. Ann.}, 306(1):547--554, 1996.

\bibitem{rubinstein2015even}
J.~H. Rubinstein and S.~Tillmann.
\newblock Even triangulations of n--dimensional pseudo-manifolds.
\newblock {\em Algebr. Geom. Topol.}, 15(5):2947--2982, 2015.

\bibitem{serre1980trees}
J.-P. Serre.
\newblock {\em Trees}.
\newblock Springer-Verlag, 1980.

\bibitem{soma1981gromov}
T.~Soma.
\newblock The gromov invariant of links.
\newblock {\em Invent. Math.}, 64(3):445--454, 1981.

\bibitem{thurston1978geometry}
W.~P. Thurston.
\newblock {\em The Geometry and Topology of Three-Manifolds}.
\newblock Lecture Notes from Princeton University, 1978--80.
\newblock \url{http://library.msri.org/books/gt3m/}.

\bibitem{waldhausen1968irreducible}
F.~Waldhausen.
\newblock On irreducible 3-manifolds which are sufficiently large.
\newblock {\em Ann. of Math.}, pages 56--88, 1968.

\bibitem{white2001diameter}
M.~E. White.
\newblock A diameter bound for closed, hyperbolic 3-manifolds.
\newblock {\em arXiv preprint}, 2001.
\newblock \url{http://arxiv.org/abs/math/0104192}.

\end{thebibliography}

\textsc{Department of Mathematics,  
Kyoto University, Kitashirakawa Oiwake-cho, 
Sakyo-ku, Kyoto 606-8502, Japan.} 

\textit{E-mail address}: \texttt{k.yoshida@math.kyoto-u.ac.jp}

\end{document}